\documentclass[twoside]{article}

\usepackage{PRIMEarxiv}

\usepackage[utf8]{inputenc} % allow utf-8 input
\usepackage[T1]{fontenc}    % use 8-bit T1 fonts
\usepackage{hyperref}       % hyperlinks
\usepackage{url}            % simple URL typesetting
\usepackage{booktabs}       % professional-quality tables
\usepackage{amsfonts}       % blackboard math symbols
\usepackage{nicefrac}       % compact symbols for 1/2, etc.
\usepackage{microtype}      % microtypography
\usepackage{lipsum}
\usepackage{mathtools}
\usepackage{multicol}
\usepackage{fancyhdr}       % header
\usepackage{graphicx}       % graphics
\graphicspath{{media/}}     % organize your images and other figures under media/ folder
\usepackage{relsize}
\usepackage{amsmath}
\usepackage{amsthm}
\usepackage{amssymb}
\usepackage{tikz}
\usepackage{tikz-cd}
\usetikzlibrary{calc}
\usepackage{faktor}
\usepackage{adjustbox}
\usepackage{varwidth}
\usepackage{tasks}
\usepackage{enumitem}

\usetikzlibrary{arrows,calc,matrix}

  \newtheorem{theorem}{Theorem}[section]
	\newtheorem{proposition}[theorem]{Proposition}
	\newtheorem{conjecture}[theorem]{Conjecture}
	\newtheorem{lemma}[theorem]{Lemma}
	\newtheorem{corollary}[theorem]{Corollary}
	\theoremstyle{definition}
	\newtheorem*{claim*}{Claim}
	\newtheorem{remark}[theorem]{Remark}
	
	\newtheorem{definition}[theorem]{Definition}
	
	\newtheorem{example}[theorem]{Example}
	
	\theoremstyle{remark}

   \DeclareMathOperator {\lk}{lk} 
   \DeclareMathOperator {\slk}{slk} 

  \DeclareMathOperator {\Ker}{Ker}

\newcommand{\abs}[1]{\left|#1\right|}
\usepackage{scalerel,stackengine}
\stackMath
\newcommand\reallywidehat[1]{%
	\savestack{\tmpbox}{\stretchto{%
			\scaleto{%
				\scalerel*[\widthof{\ensuremath{#1}}]{\kern-.6pt\bigwedge\kern-.6pt}%
				{\rule[-\textheight/2]{1ex}{\textheight}}%WIDTH-LIMITED BIG WEDGE
			}{\textheight}% 
		}{0.5ex}}%
	\stackon[1pt]{#1}{\tmpbox}%
}
\newcommand\restr[2]{{% we make the whole thing an ordinary symbol
		\left.\kern-\nulldelimiterspace % automatically resize the bar with \right
		#1 % the function
		\right|_{#2} % this is the delimiter
}}
  
\title{On the $\Sigma^1$ and $\Sigma^2$-invariants of Artin groups}
\author{Marcos Escartín-Ferrer}

\begin{document}

 \maketitle

\begin{abstract}
We prove the $\Sigma^1$-conjecture for two families of Artin groups: Artin groups such that there exists a prime number $p$ dividing $\frac{l(e)}{2}$ for every edge $e$ with even label $>2$ and balanced Artin groups. The family of balanced Artin groups extends two previously studied families: the one considered by Kochloukova in \cite{Kochloukova} and the family of coherent Artin groups. We state a conjecture on the $\Sigma^2$-invariant for Artin groups satisfying the $K(\pi,1)$-conjecture. The conjecture is proven to be true for two significant families: $2$-dimensional and coherent Artin groups. In the $2$-dimensional case we are able to compute $\Sigma^n$ for all $n\geq 2$ and to derive finiteness properties of the derived subgroup. 
\end{abstract}
The author is partially supported by the Departamento de Ciencia, Universidad y Sociedad del Conocimiento del Gobierno de Aragón (grant code: E22-23R: “Álgebra y Geometría”), and by the Spanish Government PID2021-126254NBI00
\section{Introduction}
The class of Artin groups is one of the most prominent and extensively studied families of groups in geometric group theory. These groups are investigated from algebraic, geometric, and combinatorial perspectives. An Artin group is defined as follows: given a simplicial graph $\Gamma$, i.e. a finite graph with no double edges nor loops, with vertex set $V(\Gamma)$, edge set $E(\Gamma)$ and with a labelling function $l:E(\Gamma)\to\mathbb{N}_{\geq 2}$ the associated \textbf{Artin group} is given by the following presentation:
$$A_\Gamma=\left\langle v\in V(\Gamma)\,\middle\vert\,\underset{l(e)\text{ letters}}{\underbrace{uvu\dots}}=\underset{l(e)\text{ letters}}{\underbrace{vuv\dots}}~\text{ for all }~e=\lbrace u,v\rbrace\in E(\Gamma)\right\rangle$$ 
Artin groups generalize several important families of groups, including free groups, free abelian groups, braid groups and right-angled Artin groups. A key feature of Artin groups is that many of their algebraic properties can often be described in terms of the combinatorial structure of the defining graph $\Gamma$.

Given a finitely generated group $G$, one can define a family of geometric invariants known as the $\Sigma$-invariants or Bieri-Neumann-Strebel-Renz (BNSR) invariants. These were first introduced by Bieri-Neumann-Strebel in \cite{Bieri-Neumann-Strebel} and Bieri-Renz in \cite{Bieri-Renz}. The invariants are denoted as $\Sigma^n(G)$ or $\Sigma^n(G,A)$, where $n\in\mathbb{N}$ and $A$ is a $G$-module. They are defined as subsets of the character sphere $S(G)$ associated to $G$, consisting of equivalence classes of \textbf{characters}, i.e. homomorphisms $\chi:G\to\mathbb{R}$.  The $\Sigma$-invariants encode deep information about the finiteness properties $F_n$ and $FP_n$ of the \textbf{coabelian subgroups} of $G$, i.e. subgroups $N\leq G$ such that $G/N$ is abelian.

In this paper we will study the $\Sigma$-invariants $\Sigma^1$ and $\Sigma^2$ for certain families of Artin groups. Given a character $\chi:A_\Gamma\to\mathbb{R}$, define the \textbf{living subgraph} $\mathrm{Liv}^\chi\subset\Gamma$ by deleting:
\begin{itemize}
	\item All vertices $v\in V(\Gamma)$ with $\chi(v)=0$. (called \textbf{dead vertices})
	\item The interior of every edge $e=\lbrace u,v\rbrace\in E(\Gamma)$ with $\chi(u)+\chi(v)=0$ and even label $l(e)\geq 4$. (called \textbf{dead edges})
\end{itemize}
We say that $\mathrm{Liv}^\chi$ is \textbf{dominating} if  every dead vertex $v\in V(\Gamma)$ is adjacent to some \textbf{living vertex} $w\in V(\Gamma)$, i.e. a vertex $w$ such that $\chi(w)\neq 0$. Motivated by earlier work of Meier-Meiner-VanWyk from \cite{Meier-Meiner-VanWyk} the following conjecture was proposed in \cite{Almeida}:
\begin{conjecture}\label{Conjecture 1.1}
	Let $A_\Gamma$ be an Artin group, then:
	$$\Sigma^1(A_\Gamma)=\lbrace[\chi]\in S(A_\Gamma)\mid\mathrm{Liv}^\chi\text{ is connected and dominant}\rbrace$$
\end{conjecture}
Using representation theory for finite abelian groups, we will introduce in Section \ref{Section 7} a technical criterion that allows us to prove the $\Sigma^1$-conjecture for two broad families of Artin groups:

\begin{theorem}\label{Theorem 1.1}
	The following families satisfy the $\Sigma^1$-conjecture:
	\begin{enumerate}
		\item Balanced Artin groups.
		\item Artin groups $A_\Gamma$ such that there exists a prime number $p$ dividing $\frac{l(e)}{2}$ for every $e\in E(\Gamma)$ with $l(e)>2$ even.
	\end{enumerate}
\end{theorem}
The definition of balanced Artin groups is given in Section \ref{Section 2.1}. This family generalizes two previously studied classes:
\begin{enumerate}
	\item Graphs $\Gamma$ in which $l(e)$ is even for all $e\in E(\Gamma)$ and every closed reduced path with labels strictly greater than $2$ has odd length.
	\item Coherent Artin groups (see Theorem \ref{Theorem 2.1}).
\end{enumerate}

The family in $1.$ was already known to satisfy the $\Sigma^1$-conjecture (see \cite{Kochloukova}).

In higher dimensions, Escartín-Martinez provided in \cite{Escartin} a sufficient condition for a character to lie in $\Sigma^n$ under the assumption that the Artin group satisfies the $K(\pi,1)$-conjecture. While the condition is somewhat technical, we obtain the following more tractable criterion in the case $n=2$:
\begin{theorem}\label{Theorem 1.2 } Let $A_\Gamma$ be an Artin group satisfying the $K(\pi,1)$-conjecture and let $\chi:A_\Gamma\to\mathbb{R}$ a character. Suppose:
	\begin{enumerate}
		\item For any edge $e=\lbrace v,w\rbrace\in E(\Gamma)$ such that either $\chi(v)=\chi(w)=0$ or $e$ is dead, there is some living $u\in V(\Gamma)$ such that $A_{\{v,w,u\}}$ is spherical,
		\item  For every dead vertex $v\in V(\Gamma)$, the spherical link $\mathrm{slk}_{\mathrm{Liv}^\chi,v}^\chi$ is non-empty and connected,
		\item The simplicial complex obtained from $\mathrm{Liv}^\chi$ by attaching $2$-cells along all spherical triangles except those of type $\mathbb{B}_3$ with a $3$-dead edge is $1$-acyclic (resp simply connected).
	\end{enumerate} Then $[\chi]\in\Sigma^2(A_\Gamma,\mathbb{Z})$  (resp $[\chi]\in\Sigma^2(A_\Gamma)$).
\end{theorem} 
For the definitions of the spherical link $\mathrm{slk}_{\mathrm{Liv}^\chi,v}^\chi$ and $3$-dead edges see section \ref{Section 2.4}.
 We conjecture that the converse holds:
\begin{conjecture}
	Let $A_\Gamma$ be an Artin group satisfying the $K(\pi,1)$-conjecture and let $\chi:A_\Gamma\to\mathbb{R}$ be a character. Then, $[\chi]\in\Sigma^2(A_\Gamma,\mathbb{Z})$ (resp $[\chi]\in\Sigma^2(A_\Gamma)$) if and only if the conditions of Theorem \ref{Theorem 1.2 } hold.
\end{conjecture}
We refer to this as the \textbf{(homotopical) $\Sigma^2$-conjecture for Artin groups}. We prove the conjecture for coherent and $2$-dimensional Artin groups. The characterization in the $2$-dimensional case is particularly simple:
\begin{theorem}\label{Theorem 1.4}
	Let $A_\Gamma$ be a $2$-dimensional Artin group and $0\neq\chi:A_\Gamma\to\mathbb{Z}$ a character. Then, $\Sigma^n(A_\Gamma)=\Sigma^m(A_\Gamma,\mathbb{Z})$ for all $m,n\geq 2$ and $[\chi]\in\Sigma^2(A_\Gamma)$ if and only if:
	\begin{enumerate}
		\item No edge $e=\lbrace v,w\rbrace\in E(\Gamma)$ satisfies $\chi(v)=\chi(w)=0$ or $e$ is dead,
		\item Every dead vertex is adjacent to a unique living vertex,
		\item  $\mathrm{Liv}^\chi$ is a tree.
	\end{enumerate}
\end{theorem}
As an application, we determine when the derived subgroup of a $2$-dimensional Artin group is finitely presented:
\begin{corollary}\label{Corollary 1.5}
	Let $A_\Gamma$ be a $2$-dimensional Artin group. Then $\Sigma^2(A_\Gamma)=S(A_\Gamma)$ if and only if $\Gamma$ is a tree with odd labels. In particular, $A_\Gamma'$ is finitely presented if and only if $\Gamma$ is a tree with odd labels, in which case $A_\Gamma'$ is of type $F_\infty$.
\end{corollary}
The question whether $A_\Gamma'$ is finitely generated remains open for arbitrary $2$-dimensional Artin groups.

The paper is organized as follows. In section $2$ we will review background material on Artin groups and $\Sigma$-invariants. In section $3$ we construct a chain complex for computing the homology of \textbf{Artin kernels}, i.e. kernels of characters $\chi:A_\Gamma\to\mathbb{Z}$. This chain complex serve as the main tool for analysing the $\Sigma$-invariants. Section $4$ is devoted to the $\Sigma^1$-conjecture, where we will prove Theorem \ref{Theorem 1.1}, establishing the conjecture for balanced Artin groups. In section $5$ we will focus on the $\Sigma^2$-invariant, proving Theorem \ref{Theorem 1.4} and establishing several partial converses to Theorem \ref{Theorem 1.2 }. Finally, in Section 6, we apply these results to investigate the algebraic fibration properties of some Artin groups.
\section{Preliminaries}
\subsection{Artin groups}\label{Section 2.1}
Given a simplicial graph $\Gamma$, the \textbf{Coxeter group} $W_\Gamma$ is defined similarly to the Artin group $A_\Gamma$, with the additional relation that all generators must have order $2$. It is a classical result that finite Coxeter groups are completely classified (c.f. \cite{Coxeter}). The classification of irreducible finite Coxeter groups consists of four infinite families and six sporadic cases:
\begin{multicols}{2}
	$$\mathbb{A}_n=\begin{tikzpicture}[main/.style = {draw, circle},node distance={15mm},scale=0.6, every node/.style={scale=0.5}]
		\node[main] (1) {};
		\node[main] (2) [right of=1] {}; 
		\node[main] (3) [right of=2] {}; 
		\draw[-] (1) --  (2);
		\draw[-] (2) -- (3);
		\node[main] (4) [right of=3] {};
		\node[main] (5) [right of=4] {}; 
		\node[main] (6) [right of=5] {}; 
		\draw[-] (4) -- (5);
		\draw[-] (5) -- (6);
		\node at ($(3)!.5!(4)$) {\ldots};
	\end{tikzpicture}~\forall~n\geq1$$
	$$\mathbb{B}_n=
	\begin{tikzpicture}[main/.style = {draw, circle},node distance={15mm},scale=0.6, every node/.style={scale=0.5}] 
		\node[main] (1) {};
		\node[main] (2) [right of=1] {}; 
		\node[main] (3) [right of=2] {}; 
		\draw[-] (1) -- node[above] {$\mathlarger{\mathlarger{\mathlarger{\mathlarger{4}}}}$}   (2);
		\draw[-] (2) -- (3);
		\node[main] (4) [right of=3] {};
		\node[main] (5) [right of=4] {}; 
		\node[main] (6) [right of=5] {}; 
		\draw[-] (4) --  (5);
		\draw[-] (5) -- (6);
		\node at ($(3)!.5!(4)$) {\ldots};
	\end{tikzpicture}~\forall~n\geq3$$
	$$\mathbb{D}_n=
	\begin{tikzpicture}[main/.style = {draw, circle},node distance={15mm},scale=0.6, every node/.style={scale=0.5}] 
		\node[main] (1) {};
		\node[main] (2) [right of=1] {}; 
		\node[main] (3) [right of=2] {}; 
		\node[main] (4) [above of=2] {}; 
		\draw[-] (1) -- (2);
		\draw[-] (2) -- (3);
		\draw[-] (2) -- (4);
		\node[main] (5) [right of=3] {};
		\node[main] (6) [right of=5] {}; 
		\node[main] (7) [right of=6] {}; 
		\draw[-] (5) --  (6);
		\draw[-] (6) -- (7);
		\node at ($(3)!.5!(5)$) {\ldots};
	\end{tikzpicture}~\forall~n\geq4$$
	$$\mathbb{E}_6=
	\begin{tikzpicture}[main/.style = {draw, circle},node distance={15mm},scale=0.6, every node/.style={scale=0.5}] 
		\node[main] (1) {};
		\node[main] (2) [right of=1] {}; 
		\node[main] (3) [right of=2] {}; 
		\node[main] (4) [above of=3] {}; 
		\node[main] (5) [right of=3] {}; 
		\node[main] (6) [right of=5] {}; 
		\draw[-] (1) -- (2);
		\draw[-] (2) -- (3);
		\draw[-] (3) -- (4);
		\draw[-] (3) -- (5);
		\draw[-] (5) -- (6);
	\end{tikzpicture}$$
	$$\mathbb{E}_7=
	\begin{tikzpicture}[main/.style = {draw, circle},node distance={15mm},scale=0.6, every node/.style={scale=0.5}] 
		\node[main] (1) {};
		\node[main] (2) [right of=1] {}; 
		\node[main] (3) [right of=2] {}; 
		\node[main] (4) [above of=3] {}; 
		\node[main] (5) [right of=3] {}; 
		\node[main] (6) [right of=5] {}; 
		\node[main] (7) [right of=6] {}; 
		\draw[-] (1) -- (2);
		\draw[-] (2) -- (3);
		\draw[-] (3) -- (4);
		\draw[-] (3) -- (5);
		\draw[-] (5) -- (6);
		\draw[-] (6) -- (7);
	\end{tikzpicture}$$
	$$\mathbb{E}_8=
	\begin{tikzpicture}[main/.style = {draw, circle},node distance={15mm},scale=0.6, every node/.style={scale=0.5}] 
		\node[main] (1) {};
		\node[main] (2) [right of=1] {}; 
		\node[main] (3) [right of=2] {}; 
		\node[main] (4) [above of=3] {}; 
		\node[main] (5) [right of=3] {}; 
		\node[main] (6) [right of=5] {}; 
		\node[main] (7) [right of=6] {}; 
		\node[main] (8) [right of=7] {}; 
		\draw[-] (1) -- (2);
		\draw[-] (2) -- (3);
		\draw[-] (3) -- (4);
		\draw[-] (3) -- (5);
		\draw[-] (5) -- (6);
		\draw[-] (6) -- (7);
		\draw[-] (7) -- (8);
	\end{tikzpicture}$$
	$$\mathbb{H}_3=
	\begin{tikzpicture}[main/.style = {draw, circle},node distance={15mm},scale=0.6, every node/.style={scale=0.5}] 
		\node[main] (1) {};
		\node[main] (2) [right of=1] {}; 
		\node[main] (3) [right of=2] {}; 
		\draw[-] (1) -- node[above] {$\mathlarger{\mathlarger{\mathlarger{\mathlarger{5}}}}$} (2);
		\draw[-] (2) -- (3);
	\end{tikzpicture}$$
	$$\mathbb{H}_4=
	\begin{tikzpicture}[main/.style = {draw, circle},node distance={15mm},scale=0.6, every node/.style={scale=0.5}] 
		\node[main] (1) {};
		\node[main] (2) [right of=1] {}; 
		\node[main] (3) [right of=2] {}; 
		\node[main] (4) [right of=3] {}; 
		\draw[-] (1) -- node[above] {$\mathlarger{\mathlarger{\mathlarger{\mathlarger{5}}}}$} (2);
		\draw[-] (2) -- (3);
		\draw[-] (3) -- (4);
	\end{tikzpicture}$$
	$$\mathbb{F}_4=
	\begin{tikzpicture}[main/.style = {draw, circle},node distance={15mm},scale=0.6, every node/.style={scale=0.5}] 
		\node[main] (1) {};
		\node[main] (2) [right of=1] {}; 
		\node[main] (3) [right of=2] {}; 
		\node[main] (4) [right of=3] {}; 
		\draw[-] (1) -- (2);
		\draw[-] (2) -- node[above] {$\mathlarger{\mathlarger{\mathlarger{\mathlarger{4}}}}$} (3);
		\draw[-] (3) -- (4);
	\end{tikzpicture}$$
	$$\mathbb{I}_2(k)=
	\begin{tikzpicture}[main/.style = {draw, circle},node distance={15mm},scale=0.6, every node/.style={scale=0.5}] 
		\node[main] (1) {};
		\node[main] (2) [right of=1] {}; 
		\draw[-] (1) -- node[above] {$\mathlarger{\mathlarger{\mathlarger{\mathlarger{k}}}}$} (2);
	\end{tikzpicture}~\forall~k\geq 4$$
\end{multicols}
These groups are represented via their Dynkin diagrams, which are defined in terms of the original graph $\Delta$. The Dynkin diagram has the same vertex set as $\Gamma$, and two vertices are connected by an edge if and only if the corresponding edge in $\Gamma$ has label at least $3$. By convention, edges labelled with a $3$ are unlabelled in the Dynkin diagram to simplify the notation, while non-adjacent pairs in $\Gamma$ are indicated in the Dynkin diagram with an edge labelled $\infty$. Every finite Coxeter group decomposes as a direct product of irreducible finite Coxeter groups.

In this paper, we will consider the following families of Artin groups:
\begin{enumerate}
	\item \textbf{Right-angled Artin groups (RAAGs)}: All edges of $\Gamma$ are labelled with a $2$. These are precisely the Artin groups where all the defining relations are commutators
	\item \textbf{Spherical Artin groups}: These are the Artin groups whose associated Coxeter group is finite. We say that $\Gamma$ is a \textbf{spherical graph} if the corresponding Artin group $A_\Gamma$ is spherical.
	\item \textbf{Even Artin groups}: All edges in $\Gamma$ are labelled with even integers.
	\item \textbf{Triangular Artin groups}: These are the Artin groups with exactly $3$ generators. They are denoted as $A_{MNP}$, where $M\leq N\leq P$ are the labels of the three edges of the triangle. If two vertices are not adjacent in $\Gamma$ we assume that they are connected by an edge labelled $\infty$. The only spherical triangular Artin groups are $A_{22P}$ with $2\leq P<\infty$ and $A_{23P}$ with $P\in\lbrace 3,4,5\rbrace$.
	\item \textbf{2-Dimensional Artin groups}: These are the Artin groups such that any spherical subgraph has at most $2$ vertices, i.e. the graph does not contain any spherical triangular subgraphs.
\end{enumerate}

Salvetti constructed in \cite{Salvetti} and \cite{Salvetti 2} a CW complex that plays a central role in the study of Artin groups. Due to its significance, it is now known as the \textbf{Salvetti complex}. Given an Artin group $A_\Gamma$, the associated Salvetti complex $\mathrm{Sal}(\Gamma)$ is built as the $2$-dimensional presentation complex corresponding to the standard presentation of $A_\Gamma$, with higher-dimensional cells attached for each spherical subgraph $X\subset\Gamma$. In particular the $0$-skeleton consists of a single vertex, denoted $\sigma_\emptyset$, the $1$-skeleton consists of $1$-cells $\sigma_v$ for each vertex $v\in V(\Gamma)$ and the $2$-skeleton consists of $2$-cells $\sigma_e$ for each edge $e\in E(\Gamma)$.

The most important open problem for Artin groups is the \textbf{$K(\pi,1)$-conjecture}. This conjecture admits several equivalent formulations, one of which states that the Salvetti complex $\mathrm{Sal}(\Gamma)$ is a model for the classifying space of $A_\Gamma$. In particular, the $K(\pi,1)$-conjecture holds for $A_\Gamma$ if and only if the universal cover of $\mathrm{Sal}(\Gamma)$ is  contractible (c.f. \cite{Charney-Davis}). For a detailed overview of the conjecture and its various formulations, we refer the reader to the comprehensive survey by Paris in \cite{Paris}.

We recall some basic graph-theoretical definitions. A subgraph $\Delta\subset\Gamma$ is said to be \textbf{induced} if every pair of vertices $v,w\in V(\Delta)$ that are connected by an edge in $\Gamma$ are also connected by an edge in $\Delta$. A graph $\Gamma$ is called \textbf{chordal} if it contains no induced subgraph which is a cycle of length at least four.

A group $G$ is said to be \textbf{coherent} if every finitely generated subgroup $H\leq G$ is also finitely presented. There is a well-known characterization of coherence for right-angled Artin groups. This characterization has been extended for Artin groups as follows:
\begin{theorem}[\cite{Gordon}, \cite{Wise} Gordon-Wise]\label{Theorem 2.1} Let $A_\Gamma$ be an Artin group. Then, $A_\Gamma$ is coherent if, and only if all of the following condition holds:
	\begin{enumerate}
		\item $\Gamma$ is a chordal graph.
		\item Every complete subgraph of $\Gamma$ with three or four vertices has at most one edge labelled with an integer greater than $2$.
		\item There is no induced subgraph in $\Gamma$ of the following form:
			$$
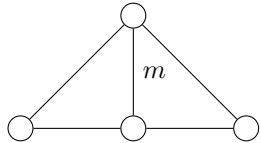
\begin{tikzpicture}[main/.style = {draw, circle},node distance={15mm}] 
			\node[main] (1) {};
			\node[main] (2) [right of=1] {}; 
			\node[main] (3) [above of=2] {}; 
			\node[main] (4) [right of=2] {}; 
			\draw[-] (1) -- (2);
			\draw[-] (2) -- node[right] {$m$} (3);
			\draw[-] (2) --  (4);
			\draw[-] (3) --  (4);
			\draw[-] (1) --  (3);
		\end{tikzpicture}$$
		where $m>2$ and all the unlabelled edges are assumed to be labelled with a $2$.
	\end{enumerate}
\end{theorem}
Finally, let us define the family of balanced Artin groups, for which we will prove that the $\Sigma^1$-conjecture holds.

\begin{definition}
	Let $\Gamma$ be a graph whose edges are labelled with even integers $\geq 4$ (not necessarily a simplicial graph). An \textbf{n-balanced colouring} of $\Gamma$ is a map $\mathcal{C}:E(\Gamma)\to\lbrace 1,\dots,n\rbrace$ such that for every $i=1,\dots,n$ there exists a prime number $p_i$ with $p _i$ dividing $\frac{l(e)}{2}$ for every edge $e$ with $\mathcal{C}(e)=i$.
\end{definition}
\begin{definition}
	Let $\Gamma$ be a graph whose edges are labelled with even integers $\geq 4$, which may not be simplicial. Suppose that $\Gamma$ admits an $n$-balanced colouring such that every closed path of even length contains exactly two or zero edges of each colour. Then, $\Gamma$ is called \textbf{balanced}.
\end{definition}
\begin{definition}
	Let $A_\Gamma$ be an Artin group. Define $\Gamma_{>2}^{\mathrm{even}}$ to be the graph obtained from $\Gamma$ by the following procedure:
	\begin{itemize}
		\item Remove all edges with a label $2$.
		\item Collapse each edge with an odd label into a single vertex.
	\end{itemize}
	We say that $A_\Gamma$ is \textbf{balanced} if the resulting graph $\Gamma_{>2}^{\mathrm{even}}$ is balanced.
\end{definition}
Note that $\Gamma_{>2}^{\mathrm{even}}$ may not be simplicial. Two vertices of $\Gamma$ are identified in $\Gamma_{>2}^{\mathrm{even}}$ if and only if they are connected by a path consisting entirely of edges with odd labels.
\subsection{$\Sigma$-invariants}
Let $G$ be finitely generated. Throughout the paper  $\chi:G\to\mathbb{R}$ will represent an arbitrary non-zero character. Two characters are said to be \textbf{equivalent}, denoted $\chi_1\sim\chi_2$, if one is a positive scalar multiple of the other, i.e. there exists $t>0$ such that $\chi_1=t\chi_2$. The \textbf{character sphere} $S(G)$ is the set of equivalence classes of non-zero characters. Note that $S(G)\cong \mathbb{S}^{r-1}$, where $r$ is the torsion free rank of the abelianization $G/G'$ and $\mathbb{S}^{r-1}$ is the $(r-1)$-sphere. A character $\chi:G\to\mathbb{Z}$ is called \textbf{discrete} and the set of all discrete characters is dense in $S(G)$ (cf. \cite{Strebel} Lemma B3.24). We now define the homological $\Sigma$-invariants of a group.
\begin{definition}
	Consider a character $\chi:G\to\mathbb{R}$ and define the monoid $G_\chi=\lbrace g\in G\mid\chi(g)\geq 0\rbrace$. Let $A$ be a left $G$-module and $n$ a positive integer. We define:
	$$\Sigma^n(G,A)=\lbrace[\chi]\in S(G)\mid A\text{ is of type }FP_n\text{ over }\mathbb{Z}G_\chi\rbrace$$
	These sets satisfy the following chain:
	$$\Sigma^\infty(G,A)\subseteq\dots\subseteq\Sigma^1(G,A)\subseteq\Sigma^0(G,A)=S(G)$$
\end{definition}
There exists also a homotopical version of the invariants: 
\begin{definition}
	Let $G$ be a group admitting a finite $K(G,1)$ and $X$ a model for $K(G,1)$. Denote its universal cover by $\widetilde{X}$. There is a bijection between the vertices of $\widetilde{X}$ and the elements of $G$. Fix a base point $b\in\widetilde{X}^{(0)}$. Given a character $\chi:G\to\mathbb{R}$ define a function $\widetilde{\chi}:\widetilde{X}^{(0)}\to\mathbb{R}$  by $\widetilde{\chi}(g\cdot b)=\chi(g)$. For $a\in\mathbb{R}$, define $\widetilde{X_{\chi}}^{[a,\infty)}$ as the maximal subcomplex of $\widetilde{X}$ with $0$-skeleton 
	$\lbrace x\in\widetilde{X}^{(0)}\mid \widetilde{\chi}(x)\geq a\rbrace$. Then, for each $d\geq 0$ the inclusion map $\widetilde{X_{\chi}}^{[0,\infty)}\hookrightarrow \widetilde{X_{\chi}}^{[-d,\infty)}$ induces a map $\tau_i^d:\pi_i\left(\widetilde{X_{\chi}}^{[0,\infty)}\right)\rightarrow \pi_i\left(\widetilde{X_{\chi}}^{[-d,\infty)}\right)$. Define the homotopical $\Sigma$-invariants as:
	$$\Sigma^n(G)=\left\lbrace[\chi]\in S(G)\mid\text{there exists }~d>0\text{ such that } \tau_i^d
	\text{  trivial }\forall~i<n\right\rbrace$$
	%$$\Sigma^n(G)=\left\lbrace[\chi]\in S(G)\mid\exists~d>0\text{ with } \pi_i\left(\widetilde{X_{\chi}}^{[0,\infty)}\right)\rightarrow \pi_i\left(\widetilde{X_{\chi}}^{[-d,\infty)}\right) 
	%\text{  trivial }\forall~i<n\right\rbrace$$
	This defines a filtration:
	$$\Sigma^\infty(G)\subseteq\dots\subseteq\Sigma^1(G)\subseteq\Sigma^0(G)=S(G)$$
\end{definition}
Renz proved in his thesis (cf. \cite{Renz}) that the homotopical $\Sigma$-invariants are well-defined, i.e. it does not depend on the chosen model for $K(G,1)$. He also proved with Bieri that if $n=\operatorname{cd}G$, where $\operatorname{cd}G$ is the cohomological dimension of $G$, then $\Sigma^m(G,\mathbb{Z})=\Sigma^n(G,\mathbb{Z})$ for any $m\geq n$. Analogously, if $n=\operatorname{gd}G$, where $\operatorname{gd}G$ is the geometric dimension of $G$, then $\Sigma^m(G)=\Sigma^n(G)$ for any $m\geq n$ (c.f. \cite{Bieri-Renz}).
Another important topological property about these invariants is that both $\Sigma^n(G,\mathbb{Z})$ and $\Sigma^n(G)$ are open in $S(G)$ (cf. \cite{Bieri-Renz}).\\

The homotopical and homological $\Sigma$-invariants are closely related. Indeed,  $\Sigma^n(G)=\Sigma^n(G,\mathbb{Z})\cap\Sigma^2(G)$ for each $n>1$ and for any commutative ring $R$, $\Sigma^n(G)\subset\Sigma^n(G,\mathbb{Z})\subset\Sigma^n(G,R)$, $\Sigma^1(G)=\Sigma^1(G,R)$   (cf. \cite{Renz}).
The reason why the $\Sigma$-invariants encode the finiteness properties of the coabelian subgroups is the following central theorem:
\begin{theorem}[\cite{Bieri-Renz} Bieri-Renz]\label{Theorem 2.3}
	Let $G$ be a group of type $F_n$ and $G'\leq N\leq G$. Then:
	\begin{enumerate}
		\item  $N$ is of type $FP_n$ if and only if $\lbrace[\chi]\in S(G)\mid\chi(N)=0\rbrace\subset \Sigma^n(G,\mathbb{Z})$.
		\item  $N$ is of type $F_n$ if and only if $\lbrace[\chi]\in S(G)\mid\chi(N)=0\rbrace\subset \Sigma^n(G)$.
	\end{enumerate}
\end{theorem}
For Artin groups, the $\Sigma$-invariants are symmetric.
\begin{lemma}[\cite{Blasco-Cogolludo-Martinez} Proposition 3.1]\label{Lemma 1.4}
	If $A_\Gamma$ is an Artin group then $\Sigma^n(A_\Gamma)=-\Sigma^n(A_\Gamma)$ and $\Sigma^n(A_\Gamma,R)=-\Sigma^n(A_\Gamma,R)$ for all $n\geq 1$ and any commutative ring $R$.
\end{lemma}
From Lemma \ref{Lemma 1.4} and Theorem \ref{Theorem 2.3}, we derive the next useful criteria for determining whether a character lies in the $\Sigma$-invariants or not.
\begin{proposition}\label{Proposition 2.4}
	Let $G$ be an Artin group satisfying the $K(\pi,1)$-conjecture, and $\chi:G\to\mathbb{Z}$ a non zero character. Let $\mathbb{F}$ be a field. If $\dim_\mathbb{F} H_n(\ker(\chi),\mathbb{F})=\infty$ then  $[\chi]\notin\Sigma^n(G,\mathbb{Z})$.
\end{proposition}
\begin{proof}
	Since $G$ satisfies the $K(\pi,1)$-conjecture the Salvetti complex is a $K(G,1)$, so $G$ is of type $F_n$. The result then follows from Theorem \ref{Theorem 2.3} and Lemma \ref{Lemma 1.4}.
\end{proof}
Another useful result that we deduce from Lemma \ref{Lemma 1.4} and Theorem \ref{Theorem 2.3} is the following:
\begin{proposition}\label{Proposition 2.8}
	Let $A_\Gamma$ be an Artin group satisfying the $K(\pi,1)$-conjecture and let $\chi:A_\Gamma\to\mathbb{Z}$ be a discrete character. Then, $\ker(\chi)$ is of type $F_n$ if and only if $[\chi]\in\Sigma^n(A_\Gamma,\mathbb{Z})$.
\end{proposition}
Finally, we mention a reduction method that simplifies the study of the $\Sigma^1$-invariants of Artin groups.
\begin{proposition}[\cite{Almeida 2} Lemma 2.14]\label{Proposition 1.6}
	Let $A_\Gamma$ be an Artin group and $\chi:A_\Gamma\to\mathbb{R}$ be a non-zero character. Let $\Gamma'$ be a graph obtained from $\Gamma$ by finitely many of the following transformations:
	\begin{enumerate}
		\item Delete a vertex $v$ such that $\chi(v)=0$, along with all incident edges.
		\item Add an edge with arbitrary even label between any pair of non-adjacent vertices $u,v\in V(\Gamma)$.
		\item Add an edge with arbitrary label between any pair of non-adjacent vertices $u,v\in V(\Gamma)$ with $\chi(u)=\chi(v)$.
		\item Given an edge $e=\lbrace u,v\rbrace$ with label $l$, replace the label by a divisor $\beta\mid l$ such that $\beta$ is even or $\chi(u)=\chi(v)$.
		\item Identify two vertices $u,v$ such that $\chi(u)=\chi(v)$. Now, for pair of edges with labels $n$ and $m$ and the same endpoints replace them with a unique edge with label $\gcd(m,n)$. Next, identify the endpoints of any edge with label equal to $1$ and remove loops. Repeat the process until obtaining a simplicial graph.
	\end{enumerate}
	Then, there exists a natural epimorphism $\pi:A_\Gamma\to A_{\Gamma'}$ and $\chi$ factors trough the induced character $\chi':A_{\Gamma'}\to\mathbb{R}$, i.e. $\chi=\chi'\circ\pi$. In particular, if $[\chi']\notin\Sigma^1(A_{\Gamma'})$ then $[\chi]\notin\Sigma^1(A_\Gamma)$. 
\end{proposition}
\subsection{The $\Sigma^1$-conjecture}
The Bieri–Neumann–Strebel invariant $\Sigma^1$ of a right-angled Artin group (RAAG) is well understood in terms of a combinatorial object known as the living subgraph.
\begin{definition}
	Let $A_\Gamma$ be a RAAG and $\chi:A_\Gamma\to\mathbb{R}$ a non-trivial character. The \textbf{living subgraph} $\mathrm{Liv}^\chi_0$ is defined to be the subgraph of $\Gamma$ induced by the set of vertices $v\in V(\Gamma)$ for which $\chi(v)\neq 0$. These are called \textbf{living vertices}, while those with $\chi(v)=0$ are called \textbf{dead vertices}.
\end{definition}
A sub-graph $\Gamma'\subset\Gamma$ is called \textbf{dominant} if for every vertex $v\in V(\Gamma\setminus\Gamma')$ there exists a vertex $w\in V(\Gamma')$ such that $\lbrace v,w\rbrace\in E(\Gamma)$.
\begin{theorem}[\cite{Meier-VanWyk} Meiner-VanWyk, 1993]\label{Theorem 1.2}
	Let $A_\Gamma$ be a RAAG. Then:
	$$\Sigma^1(A_\Gamma)=\lbrace[\chi]\in S(A_\Gamma)\mid\mathrm{Liv}^\chi_{0}\text{ is connected and dominant}\rbrace$$
\end{theorem}
For general Artin groups, the definition of living subgraph was generalized as follows:
\begin{definition}
	Let $A_\Gamma$ be an Artin group and $\chi:A_\Gamma\to\mathbb{R}$ a character. The \textbf{living subgraph}, denoted as $\mathrm{Liv}^\chi$, is obtained from $\mathrm{Liv}^\chi_0$ by deleting every edge $e=\lbrace u,v\rbrace$ such that $l(e)\geq 4$ is even and $\chi(u)+\chi(v)=0$. The deleted edges are called \textbf{dead edges} and the remaining edges are called \textbf{living edges}.
\end{definition}
This definition reduces to the previous one in the RAAG case (i.e., when all edge labels are $2$).  Somewhat counterintuitively, any edge incident to a dead vertex is automatically considered living.

The natural question arises: does Theorem \ref{Theorem 1.2} extend to arbitrary Artin groups? Partial results in this direction have been obtained. In particular, Meier proved the following:
\begin{proposition}[\cite{Meier} Meier, 1995]\label{Proposition 1.5} Let $A_\Gamma$ be an Artin group and $\chi:A_\Gamma\to\mathbb{R}$ a character. Then:
	\begin{enumerate}
		\item If $\mathrm{Liv}^\chi$ is connected and dominant then $[\chi]\in\Sigma^1(A_\Gamma)$.
		\item If $[\chi]\in\Sigma^1(A_\Gamma)$ then $\mathrm{Liv}^\chi_0$ is connected and dominant. In particular, $\mathrm{Liv}^\chi$ is dominant.
	\end{enumerate}
\end{proposition}
However, Meier was unable to show that $[\chi] \in \Sigma^1(A_\Gamma)$ implies that $\mathrm{Liv}^\chi$ is connected. Although several people worked on this problem throughout the years, it was first explicitly formulated as a conjecture in 2018. Applying Proposition \ref{Proposition 1.5}, the conjecture can be stated equivalently as follows:
\begin{conjecture}[\cite{Almeida}, Almeida]
	Let $A_\Gamma$ be an Artin group and $\chi:A_\Gamma\to\mathbb{R}$. Then:
	 $$\mathrm{Liv}^\chi\textit{ is disconnected }\Rightarrow[\chi]\notin\Sigma^1(A_\Gamma)$$
\end{conjecture}
This statement is known as the $\Sigma^1$-conjecture for Artin groups. The conjecture has been proved for several families of Artin groups:
\begin{theorem}
	Let $A_\Gamma$ be an Artin group and $\chi:A_\Gamma\to\mathbb{R}$ a non-zero character. Then, $[\chi]\in\Sigma^1(A_\Gamma)$ if and only if $\mathrm{Liv}^\chi$ is connected and dominant in the following cases:
	\begin{enumerate}
		\item $\Gamma$ is a complete graph and all edges have the same label. (\cite{Meier} Meier, 1997)
		\item $\Gamma$ is a tree. (\cite{Meier-Meiner-VanWyk 2} Meier-Meinert-VanWyk, 2001)
		\item $\Gamma$ is connected and $\pi_1(\Gamma)$ is free of rank $1$. (\cite{Almeida-Kochloukova} Almeida-Kochloukova, 2015)
		\item $\Gamma$ belongs to a certain family of complete graphs with $4$ vertices. (\cite{Almeida-Kochloukova 2} Almeida-Kochloukova, 2015)
		\item $\Gamma$ lies in a specific family of minimal graphs with $\pi_1(\Gamma)$ of arbitrary rank. (\cite{Almeida 2} Almeida, 2017)
		\item $\Gamma$ is connected and $\pi_1(\Gamma)$ is free of rank $2$. (\cite{Almeida} Almeida, 2018)
		\item $\Gamma$ is even and every closed reduced path with edge labels all greater than $2$ has odd length. (\cite{Kochloukova} Kochloukova, 2021)
		\item  There exists a prime number $p$ dividing $\frac{l(e)}{2}$ for every edge $e \in E(\Gamma)$ with $l(e) \geq 4$ even. (\cite{Escartin} Escartín-Martinez, 2023)
	\end{enumerate}
\end{theorem} 
An important simplification of the conjecture, which will be used later, is the following reduction to discrete characters.
\begin{proposition}[\cite{Almeida-Kochloukova}, Corollary 2.9]\label{Proposition 2.9}
	Let $A_\Gamma$ be an Artin group. Suppose that for every discrete character $\chi:A_\Gamma\to\mathbb{Z}$ the following holds:
	$$\left[\mathrm{Liv}_0^\chi\text{ is connected and }\mathrm{Liv}^\chi\text{ is disconnected }\right]\Rightarrow[\chi]\notin\Sigma^1(A_\Gamma),$$
	then, the $\Sigma^1$-conjecture holds for $A_\Gamma$.
\end{proposition}
\subsection{The $\Sigma^2$-conjecture}\label{Section 2.4}
To understand the $\Sigma^2$-invariant for Artin groups, we must introduce a more subtle condition than in the $\Sigma^1$ case. This requires two preliminary definitions.
\begin{definition}
	Let $A_\Gamma$ be an Artin group and $\chi:A_\Gamma\to\mathbb{R}$ a character. An edge $e=\lbrace v,w\rbrace$ is called \textbf{$3$-dead} if $l(e)=4$, $e$ belongs to a standard parabolic subgroup of type $\mathbb{B}_3$ and $2\chi(v)+\chi(w)=0$, where $w$ is the vertex adjacent to both $v$ and the edge labelled $3$ in the $\mathbb{B}_3$ subgraph.
\end{definition}
\begin{definition}
	Let $A_\Gamma$ be an Artin group and $\chi:A_\Gamma\to\mathbb{R}$ a character. 
	Let $\Gamma_1\subset\Gamma$ be a subgraph and $v\in V(\Gamma)$. The \textbf{link} $\lk_{\Gamma_1}(v)$ is the subgraph of $\Gamma_1$ induced by the set of vertices of $\Gamma_1$ that are adjacent to $v$ in $\Gamma$.
	
	 We define the \textbf{spherical link complex} $\slk_{\Gamma_1,v}^\chi$ as the simplicial complex having one $m$-cell for each $(m+1)$-clique $X\subseteq\lk_{\Gamma_1}(v)$ such that both $X$ and $X\cup\lbrace v\rbrace$ are spherical subgraphs of $\Gamma$.
\end{definition}
Using this complex, Escartín–Martínez provided a criterion in \cite{Escartin} to determine whether a character lies in $\Sigma^n(A_\Gamma)$ for Artin groups satisfying the $K(\pi,1)$-conjecture. While this condition is technical in general, for the case $n = 2$ it simplifies to a more accessible criterion, given in Theorem~\ref{Theorem 1.2}.

In fact, by combining the main result of \cite{Escartin} with the validity of the $\Sigma^1$-conjecture for tree-based Artin groups, one obtains a new proof of \cite[Theorem 5.1]{Meier-Meiner-VanWyk 2}, which asserts that if $\Gamma$ is a tree then $\Sigma^1(A_\Gamma)=\Sigma^2(A_\Gamma,\mathbb{Z})=\Sigma^2(A_\Gamma)$.

There is also another case where the $\Sigma^2$-conjecture is known to be true:
\begin{theorem}[\cite{Escartin} Theorem 7.4]\label{Theorem 4.1} Let $A_\Gamma$ be an even Artin group satisfying the $K(\pi,1)$-conjecture. Assume that there exists a prime number $p$ dividing $\frac{l(e)}{2}$ for every $e\in E(\Gamma)$ with $l(e)>2$. Then, $A_\Gamma$ satisfies the homological $\Sigma^2$-conjecture.
\end{theorem}
In Theorem~\ref{Theorem 1.4}, we extend the $\Sigma^2$-conjecture for Artin groups based on trees to the setting of $2$-dimensional Artin groups. However, it remains an open question whether the homotopical version of the $\Sigma^2$-conjecture holds in the family of groups described in the above Theorem.
\section{Homology of Artin kernels}\label{Section 3}
In this section, we introduce two chain complexes that will be essential in computing the homology of Artin groups in the next section. These complexes will serve as our main tool for studying the invariants $\Sigma^1$ and $\Sigma^2$.

Let $A_\Gamma$ be an Artin group with vertex set $V(\Gamma)=\lbrace v_0,\dots,v_n\rbrace$ and fix a total order $v_0<\cdots<v_k$ on the vertices. For each $v\in V(\Gamma)$ and each clique $X\subset\Gamma$ with $v\in V(X)$. We define $X_v=X\setminus\lbrace v\rbrace$ and the incidence number $\langle X_v \mid X \rangle \in \lbrace-1, 1\rbrace$, determined by the orientation of $X$ induced from the global ordering of $V(\Gamma)$. More precisely, suppose $V(X) = \lbrace {v_{i_0}, \dots, v_{i_k}}\rbrace$ with $v_{i_0} < \cdots < v_{i_k}$. Then, for $v_{i_\ell} \in V(X)$, we define  $\langle X_{v_{i_l}}\mid X\rangle=(-1)^l$.

\begin{proposition}(\cite{Bourbaki} Exercise 3, Chapter 4)
	Let $W_\Gamma$ be a Coxeter group and let $\Gamma'\subset\Gamma$. Then, for every coset in $W_\Gamma/W_{\Gamma'}$, there exists a unique representative $w\in W_\Gamma$ of minimal word length  in its class $wW_{\Gamma'}$.
\end{proposition}

Let $W_\Gamma$ be a Coxeter group and $\Gamma'\subset\Gamma$ a subgraph.  The elements in the coset space $W_\Gamma/W_{\Gamma'}$ that have minimal word length in their respective classes are called \textbf{$\Gamma'$-reduced} and the set of all such elements is denoted as $W_\Gamma^{\Gamma'}$. 

Assume that $W_\Gamma$ is generated by $\lbrace w_1,\dots,w_n\rbrace$ and let $A_\Gamma$ be the associated Artin group, generated by $\lbrace v_1,\dots,v_n\rbrace$. There is a natural projection homomorphism $\pi:A_\Gamma\twoheadrightarrow W_\Gamma$ defined by sending $v_i\mapsto w_i$ for all $i$. This projection admits a set-theoretic section $q: W_\Gamma\to A_\Gamma$ constructed as follows: for any $w\in W_\Gamma$, choose a word $w_{i_1}\cdots w_{i_r}$ of minimal length representing $w$. Define $q(w)=v_{i_1}\cdots v_{i_r}\in A_\Gamma$ and denote this element as $a_w$. It follows from Tits' solution to the word problem for Coxeter groups (see \cite{Charney-Davis}) that the element $a_w$ is well-defined, i.e. it does not depend on the choice of reduced expression of $w$. We define: $$T_\Gamma^{\Gamma'}=\sum_{w\in W_{\Gamma}^{\Gamma'}}(-1)^{l(w)}a_w\in\mathbb{Z}[A_\Gamma]$$
where $l(w)$ denotes the word length of $w\in W_\Gamma$.
\begin{theorem}(\cite{Davis-Leary} Section 7)\label{Davis-Leary}
	Let $A_\Gamma$ be an Artin group satisfying the $K(\pi,1)$-conjecture and let $\widetilde{\mathrm{Sal}(\Gamma)}$ denote the universal cover of the Salvetti complex of $A_\Gamma$. For every spherical subgraph $X\subset\Gamma$, let $\widetilde{\sigma}_{X}$ denote the lift of the cell $\sigma_X$ corresponding to the identity element in $A_\Gamma$. Then, the boundary maps in the chain complex $C_*(\widetilde{\mathrm{Sal}(\Gamma)})$ are given by:
	$$\partial(\widetilde{\sigma}_{X})=\sum_{v\in V(X)}\langle X_v\mid X\rangle T_X^{X_v}\widetilde{\sigma}_{X_v}.$$
\end{theorem}
Let $R$ be a commutative ring and let $\chi:A_\Gamma\to\mathbb{Z}$ be a non-trivial discrete character. By Remark 2.3 in \cite{Blasco-Cogolludo-Martinez 2}, we may assume without loss of generality that $\chi$ is surjective. Consequently, $\chi$ determines an infinite cyclic cover of $\widetilde{\mathrm{Sal}(\Gamma)}$. Given a commutative ring $R$, the $R$-chain complex of this cyclic cover can be computed as: 
$$R[t^{\pm1}]\otimes_{R[A_\Gamma]}C_*(\widetilde{\text{Sal}(\Gamma)})=R\otimes_{\ker(\chi)}C_*(\widetilde{\text{Sal}(\Gamma)})$$
where $R[t^{\pm1}]$ is regarded as a a left $R[A_\Gamma]$-module trough the action $1\ast v=t^{\chi(v)}$ for all $v\in A_\Gamma$. We denote this chain complex as $C_*(\widetilde{\text{Sal}(\Gamma)})_\chi$. The homology groups $H_\ast(\Ker(\chi),R)$ correspond to the homology of this chain complex. For a spherical subgraph $X\subset\Gamma$, we use the notation $\sigma_X^\chi$ to denote the projection of $\widetilde{\sigma}_X$ in the $\chi$-cyclic cover.

Theorem \ref{Davis-Leary}  yields the following expression of the differentials of the cyclic cover
\begin{equation}\label{differential}\partial(\sigma^\chi_{X})=\sum_{v\in V(X)}\langle X_v\mid X\rangle\chi_t( T_X^{X_v})\sigma^\chi_{X_v}\end{equation}
where:
$$\chi_t( T_X^{X_v})=\sum_{w\in W_X^{X_v}}(-1)^{l(w)}t^{\chi(a_w)}.$$

To compute the differentials in (\ref{differential}) for any spherical $X\subset\Gamma$, it is necessary to understand the $X_v$-reduced elements for every $v\in V(X)$. These elements were computed explicitly by Stumbo for all spherical irreducible Artin groups in \cite{Stumbo}. Using his formulas, one can compute the homology with field coefficients of any Artin kernel. An example of those computations can be found at the end of the section.

Let $G$ be a finite abelian group and take group epimorphism $\varphi:A_\Gamma\to G$. Define a map $\psi:A_\Gamma\to \mathbb{Z}\times G$  given by $\psi(g)=(\chi(g),\varphi(g))$. In general $\psi$ need not be surjective. However, the following criterion ensures that $\psi$ is surjective.
\begin{lemma}
	 If the restriction $\restr{\varphi}{\Ker(\chi)}$ is an epimorphism then $\psi$ is an epimorphism.
\end{lemma}
\begin{proof}
	Let $(a,b)\in \mathbb{Z}\times G$. Since $\chi$ is surjective there exists $h\in A_\Gamma$ such that $\chi(h)=a$. Then, $\psi(h)=(a,\varphi(h))$. By assumption $\restr{\varphi}{\Ker(\chi)}$ is an epimorphism, so there exists $c\in\Ker(\chi)$ such that $\varphi(c)=\varphi(h)^{-1}b$. Therefore:
	$$\psi(hc)=\psi(h)\psi(c)=(a,\psi(h))(1,\psi(h)^{-1}b)=(a,b)$$
	which shows that $\psi$ is surjective.
\end{proof}
Assume that $\psi$ is surjective. Then, $\psi$ determines an infinite cyclic cover of $\widetilde{\mathrm{Sal}(\Gamma)}$. Given a commutative ring $R$, the $R$-chain complex of this cyclic cover can be computed as:  
$$R[\mathbb{Z}\times G]\otimes_{R[A_\Gamma]}C_*(\widetilde{\text{Sal}(\Gamma)})=R\otimes_{\ker(\psi)}C_*(\widetilde{\text{Sal}(\Gamma)})$$
where $R[\mathbb{Z}\times G]$ is a $R[A_\Gamma]$-module via $1\ast v=t^{\chi(v)}\varphi(v)$. We denote this chain complex as $C_*(\widetilde{\text{Sal}(\Gamma)})_\psi$. For a spherical subgraph $X\subset\Gamma$, we use the notation $\sigma_X^\psi$ to denote the projection of $\widetilde{\sigma}_X$ in the $\psi$-cyclic cover. The homology groups $H_\ast(\Ker(\psi),R)$ correspond to the homology of this chain complex. The differentials can be computed analogously to the $\chi$-cyclic cover case, replacing each occurrence of  the substitution $t^{\chi(v)}$ by $t^{\chi(v)}\varphi(v)$.

Since $G$ is a finite group, the subgroups $\Ker(\chi)$ and $\Ker(\psi)$ are commensurable. In particular, $\Ker(\chi)$ is of type $\mathrm{FP}_n$ if and only if $\Ker(\psi)$ is of type $\mathrm{FP}_n$. This provides the following useful criterion for studying the $\Sigma$-invariants:
\begin{lemma}\label{Lemma 3.2}
	Let $G$ be a finite group, $A_\Gamma$ be an Artin group satisfying the $K(\pi,1)$-conjecture, $\psi:A_\Gamma\to \mathbb{Z}\times G$ be an epimorphism extending $\chi$ and let $\mathbb{F}$ be a field. If $\dim_\mathbb{F} H_n(\Ker(\psi),\mathbb{F})=\infty$ then  $[\chi]\notin\Sigma^n(A_\Gamma,\mathbb{Z})$.
\end{lemma}
\begin{proof}
	If $\dim_\mathbb{F} H_n(\Ker(\psi),\mathbb{F})=\infty$  then $\Ker(\psi)$ is not $\mathrm{FP}_n$ which implies that $\Ker(\chi)$ is not $\mathrm{FP}_n$. Hence, $[\chi]\notin\Sigma^n(A_\Gamma,\mathbb{Z})$ by Proposition \ref{Proposition 2.4}.
\end{proof}
\begin{remark}
	To apply the Lemma, the $K(\pi,1)$-conjecture is assumed to ensure that $A_\Gamma$ is of type $F_n$. Since every Artin group is finitely presented, both Lemma \ref{Lemma 3.2} and Proposition \ref{Proposition 2.4} hold for $n\leq 2$ without the $K(\pi,1)$-assumption.
	
	More precisely, the $2$-skeleton of the Salvetti complex is the presentation complex of $A_\Gamma$, giving a natural way of computing the first homology group. Hence, the criterion can be applied for $n=1$ without the $K(\pi,1)$-assumption. However, without the $K(\pi,1)$-assumption, it is not known whether the 3-skeleton of the Salvetti complex suffices to compute the second homology group. Therefore, the results without assuming the $K(\pi,1)$-conjecture will only be used for the $\Sigma^1$-case.
\end{remark}
We are interested in computing the first two homology groups $H_1$ and $H_2$, which requires understanding the differentials of cells associated to cliques $X\subset\Gamma$ with $\abs{ V(X)}\leq 3$.  The relevant differentials are the following:
\begin{proposition}\label{FormulasDiferential}
	Let $A_\Gamma$ be a triangular Artin group generated by $a,b$ and $c$ and fix a total order $a<b<c$ on the vertex set. Then, the differentials in the chain complex $C_*(\widetilde{\text{Sal}(\Gamma)})_\chi$ are given as follows: 
\begin{center}
	\begin{varwidth}{\textwidth}
		\begin{enumerate}[label=(\Roman*)]
			\item\label{Formula1}  $\partial(\sigma_v^\chi)=(1-t^{\chi(v)})\sigma_\emptyset^\chi~\qquad \qquad \qquad \qquad \qquad \qquad \qquad \qquad \qquad \qquad \qquad \qquad\qquad \,\,\,\,\,\,\,\forall~v\in V(\Gamma)$
			\item\label{Formula2}  $\partial(\sigma_{ab}^\chi)=[(1-t^{\chi(a)})\sigma_b^\chi-(1-t^{\chi(b)})\sigma_a^\chi]\frac{t^{k(\chi(a)+\chi(b))}-1}{t^{\chi(a)+\chi(b)}-1}\qquad \qquad \qquad\qquad\qquad\quad\,\,\,\,\text{ if }l(\{a,b\})=2k$
			\item\label{Formula3}$\partial(\sigma_{ab}^\chi)=[\sigma_b^\chi-\sigma_a^\chi]\frac{t^{k\chi(a)}+1}{t^{\chi(a)}+1}\qquad \qquad \qquad \qquad \qquad \qquad \qquad \qquad \qquad  \qquad\,\,\,\,\,\text{ if }l(\{a,b\})=k\text{ is odd}$
			\item \label{Formula4}$	\partial(\sigma_{abc}^\chi)=(1-t^{\chi(a)})\sigma_{bc}^\chi+[(1-t^{\chi(c)})\sigma_{ab}^\chi-(1-t^{\chi(b)})\sigma_{ac}^\chi]\frac{(t^{\chi(b)+\chi(c)})^k-1}{t^{\chi(b)+\chi(c)}-1}\qquad\quad\,\,\,\,\,\text{if $A_\Gamma=A_{\{2,2k,2\}}$}$
			%%\item \label{Formula 6} $\partial(\sigma_{abc}^\chi)=(1-t^{\chi(a)})\sigma_{bc}^\chi+[\sigma_{ab}^\psi-\sigma_{ac}^\psi]\frac{1+(t^{\chi(b)})^k}{1+t^{\chi(b)}}\qquad\qquad\qquad\qquad\qquad\,\text{if $A_\Gamma=A_{\{2,k,2\}}$ with $k$ odd}$
			\item\label{Formula5}	$\partial(\sigma_{abc}^\chi)=(1-t^{\chi(a)})(1+t^{\chi(a)+\chi(b)})(1-t^{\chi(a)+2\chi(b)})\sigma_{bc}^\chi-\qquad\qquad\qquad\qquad\quad\,\,\,\,\,\,\text{if $A_\Gamma=A_{\{4,3,2\}}$}\\			\text{}\qquad\qquad\,\,\,-(1+t^{\chi(a)+\chi(b)})(1-t^{\chi(a)+2\chi(b)})(t^{2\chi(a)}-t^{\chi(a)}+1)\sigma_{ac}^\chi+\\
			\text{}\qquad\qquad\,\,+(1-t^{\chi(a)+2\chi(b)})(t^{2\chi(a)}-t^{\chi(a)}+1)\sigma_{ab}^\chi$
		\end{enumerate}
	\end{varwidth}
\end{center}
Here, the notation $A_\Gamma=A_{P,Q,M}$ means that the defining labels of $\Gamma$ are $l(\lbrace a,b\rbrace)=P,l(\lbrace b,c\rbrace)=Q$ and $l(\lbrace c,a\rbrace)=M$. 
\end{proposition}
\begin{proof}
	We illustrate the argument by proving (II); the remaining cases follow analogously using Stumbo's formulas. From equation \ref{differential}, we have:
	$$\partial(\sigma_{ab}^\chi)=\chi_t(T^b_{ab})\sigma_b^\chi-\chi_t(T^a_{ab})\sigma_a^\chi$$
	By \cite{Stumbo}, the set of $a$-reduced elements is $W^b_{ab}=\lbrace 1,a,ab,aba,\dots,(ab)^{k-1}a\rbrace$, which yields $T^b_{ab}=1-a+ab-\cdots-(ab)^{k-1}a$. Hence, applying $\chi_t$ we obtain:
	$$\chi_t(T^b_{ab})=1-t^{\chi(a)}+t^{\chi(a)+\chi(b)}-\cdots-t^{k\chi(a)+(k-1)\chi(b)}=(1-t^{\chi(a)})\frac{t^{k(\chi(a)+\chi(b))}-1}{t^{\chi(a)+\chi(b)}-1}$$
	A similar calculation gives the formula for $\chi_t(T^a_{ab})$ , completing the proof of (II)
\end{proof}
Recall that the differentials of $C_*(\widetilde{\text{Sal}(\Gamma)})_\psi$ are obtained from those of $C_*(\widetilde{\text{Sal}(\Gamma)})_\psi$  by applying the substitution $t^{\chi(v)}\mapsto t^{\chi(v)}\varphi(v)$ for all $v\in V(\Gamma)$. For example, if $l(\{a,b\})=2k$, the differential is given by:
$$\partial(\sigma_{ab}^\psi)=[(1-\varphi(a)t^{\chi(a)})\sigma_b^\psi-(1-\varphi(b)t^{\chi(b)})\sigma_a^\psi]\frac{\left(\varphi(ab)t^{\chi(a)+\chi(b)}\right)^k-1}{\varphi(ab)t^{\chi(a)+\chi(b)}-1}$$
\section{The $\Sigma^1$-conjecture}
\subsection{The $\Sigma^1$-conjecture and characters of finite groups}\label{Section 7}
Using Lemma \ref{Lemma 3.2} we now prove an equivalent formulation of the $\Sigma^1$-conjecture, reducing the problem to a question in character theory. 
This technical result will serve as the main result used in the proof of Theorem \ref{Theorem 1.1}.

We begin by simplifying the setting of the $\Sigma^1$-conjecture. Let $A_\Gamma$ be an Artin group and $\chi:A_\Gamma\to\mathbb{Z}$ be a non-trivial character. By Proposition \ref{Proposition 2.9} we may assume without loss of generality that $\mathrm{Liv}_0^\chi$ is connected, $\mathrm{Liv}^\chi$ is disconnected and that is enough to prove that $[\chi]\notin\Sigma^1(A_\Gamma)$.

Let $G$ be a finite abelian group and assume that there exists an epimorphism $\psi: A_\Gamma\to \mathbb{Z}\times G$ extending $\chi$. Using Proposition \ref{Proposition 1.6}.1 and passing to a special subgroup if necessary, we may further assume that $\chi(v)\neq0$ for all $v\in V(\Gamma)$, in particular $\Gamma=\mathrm{Liv}_0^\chi$.

Now, suppose that $\mathrm{Liv}^\chi$ has $n\geq 2$ connected components, denoted by $\mathrm{Liv}^\chi_1,\dots,\mathrm{Liv}^\chi_n$. Fix some $i\in\lbrace 1,\dots,n-1\rbrace$ and let $\Gamma_1$ and $\Gamma_2$ be the full subgraphs of $\Gamma$ generated by the vertex sets: 
$$V(\Gamma_1)=V(\mathrm{Liv}^\chi_1\cup\cdots\cup \mathrm{Liv}^\chi_{i}), V(\Gamma_2)=V(\mathrm{Liv}^\chi_{i+1}\cup\cdots\cup \mathrm{Liv}^\chi_n)$$
Then, $\Gamma=\Gamma_1\cup\Gamma_2$ and all the edges connecting $\Gamma_1$ with $\Gamma_2$ are dead. Using the transformations from Proposition \ref{Proposition 1.6}.2 and Proposition \ref{Proposition 1.6}.4 we construct a new graph $\Gamma'$ as follows:
\begin{itemize}
	\item Add an edge labelled $2$ between any pair of non-adjacent vertices in $\Gamma_1$ and in $\Gamma_2$.
	\item Replace every even label $\geq 4$ by a $2$ within each  of $\Gamma_1$ and $\Gamma_2$.
\end{itemize}
The resulting graph $\Gamma'$ is the union of two disjoint complete subgraphs $\Gamma_1'$ and $\Gamma_2'$ such that all labels in each subgraph are either $2$ or odd integers and  all edges between $\Gamma_1'$ and $\Gamma_2'$ are dead.

Let $\pi: A_\Gamma\to A_{\Gamma'}$ be the natural projection map and let $\chi'$ be the character induced by $\chi$. Then, by Proposition \ref{Proposition 1.6}, we have that if $[\chi']\notin\Sigma^1(A_{\Gamma'})$, then $[\chi]\notin\Sigma^1(A_{\Gamma})$. Therefore, it suffices to show that $[\chi']\notin\Sigma^1(A_{\Gamma'})$. To do that we will compute $H_1(\ker(\psi'),\mathbb{F})$ and use Lemma \ref{Lemma 3.2}, where $\psi'$ is the map induced by $\psi$ on $A_{\Gamma'}$. To simplify the notation in the remaining of the section we will write $A_\Gamma,\chi,\varphi,\psi$ instead of $A_{\Gamma'},\chi',\varphi',\psi'$.

To simplify the computation of homology, we define a normalization of the chain complex using a set of coefficients $p_X^\psi$. It requires to work over a ring in which those coefficients are invertible. Let $\mathbb{F}$ be a field and consider the group ring $\mathbb{F}[\mathbb{Z}\times G]$. We can interpret it as either the group ring $R[G]$ with $R=\mathbb{F}[t^{\pm1}]$ or as the polynomial ring $\mathbb{F}G[t^{\pm1}]$. Define the following multiplicative set:
$$S=\left\lbrace  zgt^{n}+\sum_{i<n}\alpha_it^i \left| n\in\mathbb{Z},z\in\mathbb{F}\setminus\lbrace 0\rbrace ,g\in G\text{ and }\alpha_i\in\mathbb{F}G\text{ with  only finitely many }\alpha_{i}\neq0\right.\right\rbrace$$
That is, $S$ is the subset of the polynomial ring $\mathbb{F}G[t^{\pm1}]$ consisting of all polynomials whose leading coefficient is a non-zero scalar multiple of a group element $g\in G$.
\begin{claim*}
	$S$ is a multiplicative set with no zero divisors.
\end{claim*}
\begin{proof}
	If $x,y\in S$ then:
	$$xy=\left(z_xg_xt^{n}+\sum_{i<n}\alpha_{x,i}t^i\right)\left(z_yg_yt^{m}+\sum_{i<m}\alpha_{y,i}t^i\right)=z_xz_yg_xg_yt^{n+m}+\sum_{i<n+m}\alpha_{xy,i}t^i$$
	Hence, the leading term of $xy$ is $z_xz_yg_xg_y\neq0$, so $0\neq xy\in S$. 
\end{proof}
Now, localizing the chain complex of the cyclic cover $R[G]\otimes_{\mathbb{F}[A_\Gamma]}C_*(\widetilde{\text{Sal}(\Gamma)})$ with respect to the multiplicative set $S$, we obtain a new complex $(R[G])_S\otimes_{\mathbb{F}[A_\Gamma]}C_*(\widetilde{\text{Sal}(\Gamma)})$. In the localized complex we normalize the generators by setting $\widetilde{\sigma}_X^\psi=\frac{1}{p_X^\psi}\sigma_X^\psi$ for each cell $\sigma_X^\psi$ with $\abs{V(X)}\leq 2$ and $\widetilde{\sigma}_X^\psi=\sigma_X^\psi$ otherwise. We define the normalizing elements $p_X^{\psi}$ as follows:
\begin{itemize}
	\item $p_\emptyset^\psi=1$.
	\item If $v\in V(\Gamma)$: $p_v^{\psi}=1-t^{\chi(v)}\varphi(v)$.
	\item If $l(\{a,b\})=2k$ is even: $p_{ab}^{\psi}=(1-t^{\chi(a)}\varphi(a))(1-t^{\chi(b)}\varphi(b))$.
	\item If $l(\{a,b\})=k$ is odd: $p_{ab}^{\psi}=(1-t^{\chi(a)}\varphi(a))\frac{(t^{\chi(a)}\varphi(a))^k+1}{t^{\chi(a)}\varphi(a)+1}$.
\end{itemize}
Each $p^{\psi}_X$ lies in the multiplicative set $S$ for every $X\subset\Gamma$ spherical with $\abs{V(X)}\leq 2$, so the normalized chains $\widetilde{\sigma}_X^\psi$ are well-defined in the localized complex, which we denote by $\widetilde{C}_\ast(\mathrm{Sal}_\Gamma^\psi)$. Using the Formulas \ref{Formula1},\ref{Formula2} and \ref{Formula3} from Proposition \ref{FormulasDiferential} we obtain:
\begin{itemize}
	\item If $X=\lbrace v\rbrace$ then $\partial(\widetilde{\sigma}_X^\psi)=\widetilde{\sigma}_\emptyset^\psi$.
	\item  	If $X=\mathbb{I}_2(k)=\lbrace a,b\rbrace$ with $k=2$ or odd, then $\partial(\widetilde{\sigma}_X^\psi)=\widetilde{\sigma}_{a}^\psi-\widetilde{\sigma}_{b}^\psi$.
	\item If $X=\mathbb{I}_2(2k)=\lbrace a,b\rbrace$ with $k>1$ then $\partial(\widetilde{\sigma}_X^\psi)=[\widetilde{\sigma}_{a}^\psi-\widetilde{\sigma}_{b}^\psi]\frac{1-\varphi(ab)^k}{1-\varphi(ab)}$.
\end{itemize}
\begin{lemma}\label{Lemma 4.1}
	The homology group $H_n(\Ker(\psi),\mathbb{F})$ is an $RG$-module whose free part has the same rank as the $(R[G])_S$-dimension of the $n$-th homology group of the normalized complex $\widetilde{C}_\ast(\mathrm{Sal}_\Gamma^\psi)$. Therefore, the homology group $H_n(\Ker(\psi),\mathbb{F})$ has finite $\mathbb{F}$-dimension if and only if the $n$-th homology group of the normalized chain complex $\widetilde{C}_\ast(\mathrm{Sal}_\Gamma^\psi)$ vanishes. 
\end{lemma}
\begin{proof}
	Consider the chain complex $C_\ast = R[G] \otimes_{\mathbb{F}[A_\Gamma]} C_\ast(\widetilde{\mathrm{Sal}}(\Gamma))$. Localize this complex at the multiplicative set $S \subset R[G]$. Since localization is a flat operation, the homology of the localized complex satisfies $H_n((R[G])_S \otimes_{RG} C_\ast) \cong (R[G])_S \otimes_{RG} H_n(C_\ast)$.
	This implies that the rank over $(R[G])_S$ of the localized homology equals the rank of the free part of $H_n(C_\ast)$ over $R[G]$, i.e.:
	$$\operatorname{rank}_{(R[G])_S} H_n(\widetilde{C}_\ast(\mathrm{Sal}_\Gamma^\psi)) = \operatorname{rank}_{R[G]} H_n(\ker(\psi), \mathbb{F})_{\text{free}}$$
	Since the torsion part of an $RG$-module becomes trivial after localizing at $S$, it follows that the homology group $H_n(\ker(\psi), \mathbb{F})$ has finite $\mathbb{F}$-dimension if and only if $H_n(\widetilde{C}_\ast(\mathrm{Sal}_\Gamma^\psi)) = 0$.
\end{proof}

In view of Lemma \ref{Lemma 4.1}, it suffices to compute the homology group $H_1(\Ker(\psi),\mathbb{F})=\Ker(\partial_1)/\mathrm{Im}(\partial_2)$ of the normalized chain complex $\widetilde{C}_\ast(\mathrm{Sal}_\Gamma^\psi)$. 
Choose a directed spanning tree $T$ of $\Gamma$. Then, the kernel of $\partial_1$ is equal to: 
$$\ker(\partial_1)=\bigoplus\limits_{\lbrace v,w\rbrace\in E(\Gamma)}(RG)_S(\widetilde{\sigma}_v^\psi-\widetilde{\sigma}_w^\psi)$$
The image of $\mathrm{Im}(\partial_2)$ depends on whether an edge is living or dead: for each living edge $e=\lbrace v,w\rbrace$, the image includes $\widetilde{\sigma}_v^\psi-\widetilde{\sigma}_w^\psi$; while for each dead edge $e=\lbrace v,w\rbrace$ with label $l(e)=2k$, the image includes $\frac{1-\varphi(vw)^k}{1-\varphi(vw)}\left(\widetilde{\sigma}_v^\psi-\widetilde{\sigma}_w^\psi\right)$. Since each connected component $\Gamma_i$ of $\mathrm{Liv}^\chi$ consists only of living edges, we have $[\widetilde{\sigma}_v^\psi]=[\widetilde{\sigma}_w^\psi]$ in the homology for all $v,w\in V(\Gamma_i)$ and $i=1,2$. Therefore, the first homology group of the chain complex $\widetilde{C}_\ast(\text{Sal}_\Gamma^\chi)$ is isomorphic to the quotient of $(R[G])_S$ by the submodule generated by the coefficients of the torsion terms coming from the dead edges. In particular, if we define $I\subset (R[G])_S$ to be the ideal generated by the elements $\frac{1-\varphi(vw)^k}{1-\varphi(vw)}=1+\varphi(vw)+\dots+\varphi(vw)^{\frac{l(e)}{2}-1}$, where $e=\lbrace v,w\rbrace$ ranges over all dead edges and $l(e)=2k$, then:

$$H_1\left(\widetilde{C}_\ast(\text{Sal}_\Gamma^\chi),\mathbb{F}\right)\neq0\Leftrightarrow I\subsetneq (R[G])_S$$
Let us study the problem of determining when the ideal $I$ is equal to $(R[G])_S$ for the case where $\mathbb{F}=\mathbb{C}$.
\begin{lemma}
	Let $G$ be a finite abelian group and let $\mathcal{P}=\lbrace p_i\rbrace_{i=1}^k$ be a set of elements of $\mathbb{C}[G]$. The ideal $I$ generated by $\mathcal{P}$ is a proper subset of $\mathbb{C}[G]$ if and only if there exists an irreducible character $\mu:G\to \mathbb{C}^\times$ such that $\mu(p_i)=0~\forall~i=1,\dots,k$.
\end{lemma}
\begin{proof}
	Since $G$ is finite abelian,  all of its finite-dimensional irreducible complex representations are $1$-dimensional. Consequently, we have a group ring decomposition $\mathbb{C}[G]\cong\mathbb{C}e_1\oplus\dots\oplus \mathbb{C}e_n$,
	 where $\lbrace e_i\rbrace_{i=1}^n$ are the orthogonal idempotents corresponding to the irreducible characters $\mu_i$,  and $\sum\limits_{i=1}^ne_i=1$. Let $\mu$ be the irreducible character associated to a given idempotent $e$. Then, for any element $\lambda=\sum\limits_{g\in G}\lambda_gg\in\mathbb{C}[G]$, we obtain:
	$$\lambda e=\sum\limits_{g\in G}\lambda_g ge=\sum\limits_{g\in G}\lambda_g\mu(g)e=\mu\left(\sum_{g\in G}\lambda_gg\right)e=\mu(\lambda)e$$
	This shows that $\lambda e=0$ if and only if $\mu\left(\lambda\right)=0$. Therefore, there exists an irreducible character $\mu:G\to\mathbb{C}^\times$ such that $\mu(p_i)=0~\forall~i=1,\dots, k$ if and only if there is an idempotent $e$ with $Ie=0$. In this case, since $\mathbb{C}[G]=\mathbb{C}[G]e\oplus\mathbb{C}[G](1-e)$ we conclude that $I\lneq\mathbb{C}[G]$.

	Conversely, suppose that no such character $\mu$ exists. Then, $Ie_i\neq 0$ for all $i=1,\dots,n$. Since each $\mathbb{C}e_i$ is one dimensional, we conclude that
		$$I=Ie_1\oplus\cdots\oplus Ie_n=\mathbb{C}e_1\oplus\cdots\oplus \mathbb{C}e_n=\mathbb{C}[G]$$
	as desired.
\end{proof}
\begin{corollary}\label{Corllary4.3}
	Let $G$ a finite abelian group and $\mathcal{P}=\lbrace p_i\rbrace_{i=1}^k$ be a set of elements of $(R[G])_S$. Then, the ideal $I$ generated by $\mathcal{P}$ is distinct from $(R[G])_S$ if and only if there exists an irreducible character $\mu:G\to \mathbb{C}^\times$ such that $\mu(p_i)=0~\forall~i=1,\dots,k$ .
\end{corollary}
\begin{proof}
	Let $Q(R)$ denote the field of fractions of $R$. Then, we have inclusions $\mathbb{C}[G]\hookrightarrow Q(R)[G]\hookrightarrow (R[G])_S$. As in the proof of the lemma, the decomposition  $\mathbb{C}[G]\cong\mathbb{C}e_1\oplus\dots\oplus \mathbb{C}e_n$ induces a decomposition:
	 $$Q(R)[G]\cong Q(R)e_0\oplus\dots\oplus Q(R)e_n$$
	 and tensoring with  $(R[G])_S$ we obtain:
	 $$(R[G])_S\cong (R[G])_Se_0\oplus\dots\oplus (R[G])_Se_n$$
	 Since localization is exact and preserves direct sums, the same argument as in the previous Lemma shows that $I\neq (R[G])_S$ if and only if there exists some idempotent $e$ with $Ie=0$, which is equivalent to the existence of a character $\mu$ such that $\mu(p_i)=0$ for all $i$.
\end{proof}
Now, we are ready to prove the main result of this section.
\begin{theorem}\label{Corollary 4.7}
	Let $A_\Gamma$ be an Artin group and let $\chi:A_\Gamma\to\mathbb{Z}$ be a non-zero discrete character such that $\mathrm{Liv}^\chi_0$ is connected and $\mathrm{Liv}^\chi$ is disconnected. Let $\Gamma_1$ and $\Gamma_2$ be unions of connected components of $\mathrm{Liv}^\chi$ such that $\Gamma_1\cup\Gamma_2=\mathrm{Liv}^\chi$ and $\Gamma_1\cap\Gamma_2=\emptyset$. Suppose that there exists a finite group $G$, an homomorphism $\varphi:A_\Gamma\to G$ such that the restriction $\restr{\varphi}{\Ker(\chi)}$ is an epimorphism and a character $\mu:G\to\mathbb{C}^\times$ such that for each dead edge  $e=\lbrace u,v\rbrace\in E(\Gamma)$ with $u\in V(\Gamma_1)$ and $v\in V(\Gamma_2)$ the value $\mu(\varphi(uv))$ is a non-trivial $\frac{l(e)}{2}$-th root of unity. Then, $[\chi]\notin\Sigma^1(A_\Gamma)$.
\end{theorem} 
\begin{proof}
	Applying the procedure outlined at the beginning of the section, we may assume that $\Gamma_1$ and $\Gamma_2$ are two complete graphs with no even labels $\geq 4$. In this situation we have that $H_1\left(\widetilde{C}_\ast(\text{Sal}_\Gamma^\chi),\mathbb{F}\right)\neq0$ if and only if the ideal $I$ generated by the elements $p_e=1+\varphi(uv)+\dots+\varphi(uv)^{\frac{l(e)}{2}-1}$ for all dead edges $e=\lbrace u,v\rbrace\in E(\Gamma)$ with $u\in V(\Gamma_1)$ and $v\in V(\Gamma_2)$ is a proper subset of $(R[G])_S$, i.e. $I\subsetneq (RG)_S$. For every such edge $e$ we have:
	$$\mu(p_e)=\mu\left(1+\varphi(uv)+\cdots+\varphi(uv)^{\frac{l(e)}{2}-1}\right)=1+\mu(\varphi(uv))+\cdots+\mu(\varphi(uv))^{\frac{l(e)}{2}-1}$$
	Since $\mu(\varphi(uv))$ is a $\frac{l(e)}{2}$-root of unity, the sum vanish, i.e. $\mu(p_e)=0$. Hence, by Corollary \ref{Corllary4.3}, it follows that $I\subsetneq (RG)_S$. Finally, applying Lemma \ref{Lemma 4.1} and Lemma \ref{Lemma 3.2}, we conclude that $[\chi]\notin\Sigma^1(A_\Gamma)$.
\end{proof}
To prove Theorem \ref{Theorem 1.1} we need the following version of the above Theorem:
\begin{corollary}\label{CorollaryArtinColor}
	Let $A_\Gamma$ be an Artin group and let $\chi:A_\Gamma\to\mathbb{Z}$ be a non-zero discrete character such that $\mathrm{Liv}^\chi_0$ is connected and $\mathrm{Liv}^\chi$ is disconnected. Let $\Gamma_1$ and $\Gamma_2$ be unions of connected components of $\mathrm{Liv}^\chi$ such that $\Gamma_1\cup\Gamma_2=\mathrm{Liv}^\chi$ and $\Gamma_1\cap\Gamma_2=\emptyset$. Let $\Lambda\subset\Gamma_{>2}^{\mathrm{even}}$ be the subgraph generated by the edges with one endpoint in $\Gamma_1\cap \mathrm{Liv}^\chi_0$ and the other endpoint in $\Gamma_2\cap \mathrm{Liv}^\chi_0$. Suppose that $\mathcal{C}$ is an $n$-balanced colouring of $\Lambda$, where each colour has an associated prime numbers $p_1,\dots,p_n$. Assume that there is a homeomorphism $\varphi: A_\Gamma\to C_{p_1}\times\dots\times C_{p_n}$ such that for each edge  $e=\lbrace u,v\rbrace\in E(\Lambda)$ we have $\varphi(uv)\in\lbrace g_{\mathcal{C}(e)}, g_{\mathcal{C}(e)}^{-1}\rbrace$, where $g_i$ is the generator of $C_{p_i}$. Then, $[\chi]\notin\Sigma^1(A_\Gamma)$.
\end{corollary}
\begin{proof}
	Since for each $e=\lbrace u,v\rbrace\in E(\Lambda)$ we have $\chi(uv)=0$ it follows that $\restr{\varphi}{\Ker(\chi)}$ is an epimorphism. Define a character $\mu: C_{p_1}\times\dots\times C_{p_n}\to \mathbb{C}^\times$ by setting $\mu(g_i)=e^{\frac{2\pi i}{p_i}}$ for each $i=1,\dots,n$. Then, for each edge $\lbrace u,v\rbrace\in E(\Lambda)$ we have $\mu(\varphi(uv))\in \lbrace  e^{{\frac{2\pi i}{p_{\mathcal{C}(e)}}}} e^{-{\frac{2\pi i}{p_{\mathcal{C}(e)}}}}\rbrace$. Because each $p_{\mathcal{C}(e)}$ divides $\frac{l(\lbrace u,v\rbrace)}{2}$ it follows that $\mu(\varphi(uv))$ is a non-trivial $\frac{l(\lbrace u,v\rbrace)}{2}$-root of unity. Hence, we can apply Theorem \ref{Corollary 4.7} and $[\chi]\notin\Sigma^1(A_\Gamma)$.
\end{proof}
\subsection{Parity of balanced Artin groups}
In this section we will introduce the notion of parity of pairs of edges, which will be a technical tool needed to prove Theorem \ref{Theorem 1.1}.

Consider a balanced graph $\Gamma$ with no closed paths of odd length equipped with an $n$-balanced colouring $\mathcal{C}$. For each $i\in\lbrace 1,\dots,n\rbrace$, denote by $E_i(\Gamma)$ the set of edges of colour $i$ in $\Gamma$. Define a relation $\sim$ on $E_i(\Gamma)$ as follows: for edges $e,f\in E_i(\Gamma)$, say that $e\sim f$ if there exists a finite sequence of minimal simple cycles $\mathcal{P}_1,\dots,\mathcal{P}_m$ such that $e\in\mathcal{P}_1,f\in\mathcal{P}_m$ and for each $k=1,\dots,m-1$ the cycles $\mathcal{P}_k$ and $\mathcal{P}_{k+1}$ share at least one edge of colour $i$.
It follows that $\sim$ is an equivalence relation on $E_{i}(\Gamma)$, partitioning the set of edges of colour $i$ into equivalence classes.
\begin{definition}
	Let $\Gamma$ be a balanced graph with no closed paths of odd length. Let $\mathcal{C}$ be an $n$-balanced colouring of $\Gamma$. Suppose $e_1,e_2\in E(\Gamma)$
	are two distinct edges of the same colour, i.e. $\mathcal{C}(e_1)=\mathcal{C}(e_2)$, and that they lie on the same equivalence class. Let $\mathcal{P}$ be a simple cycle containing both $e_1$ and $e_2$. We define the \textbf{parity} of the pair $(e_1,e_2)$ in $\mathcal{P}$ as:
	$$\mathrm{parity}_\mathcal{P}(e_1,e_2)=(-1)^l$$
	where $l$ is the numbers of edges lying between $e_1$ and $e_2$ along the cycle $\mathcal{P}$. 
\end{definition}

\begin{lemma}
	The parity map is well-defined and does not depend on the choice of the simple cycle.
\end{lemma}
\begin{proof}
	Suppose $e_1,e_2\in E(\Gamma)$
	are two edges of the same colour, that there exists a simple cycle $\mathcal{P}$ of even length containing both $e_1$ and $e_2$ and that $l$ is the number of intermediate edges between $e_1$ and $e_2$ in $\mathcal{P}$. Since $\mathcal{P}$ has even length $L$, the complementary arc from $e_1$ to $e_2$ has length $L-l-2$ and $(-1)^l=(-1)^{L-l-2}$. Which shows that the parity does not depend on the orientation of $\mathcal{P}$.
	
	Suppose that $\mathcal{P}$ and $\mathcal{Q}$ are two simple cycles of even length both containing $e_1$ and $e_2$. Denote by $l_\mathcal{P}$ and $l_{\mathcal{Q}}$the number of edges between $e_1$ and $e_2$ along $\mathcal{P}$ and $\mathcal{Q}$, respectively. Consider the walk $\mathcal{R}$ obtained by traversing from $e_1$ to $e_2$ along $\mathcal{P}$ and then $\mathcal{Q}$ from $e_2$ to $e_1$. This is a closed path of length $l_\mathcal{P}+l_{\mathcal{Q}}$, which is an even number by hypothesis. Hence $(-1)^{l_\mathcal{P}}=(-1)^{l_{\mathcal{Q}}}$, showing that the parity is independent of the chosen simple cycle. 
\end{proof}
Since the parity does not depend on the choice of the simple cycle we will omit the path and denote it as $\mathrm{parity}(e_1,e_2)$.

\begin{lemma}\label{ParityLemma} Let $\Gamma$ be a balanced graph with no closed paths of odd length and let $\mathcal{C}$ be an $\mathrm{n}$-balanced colouring. Then, there exists a function $\mathrm{p}:E(\Gamma)\to\lbrace 1,-1\rbrace$ such that for any pair of edges $e_1,e_2\in E(\Gamma)$ of the same colour that lie on the same equivalence class we have:
	$$\mathrm{parity}(e_1,e_2)=\mathrm{p}(e_1)\mathrm{p}(e_2)$$
\end{lemma}
\begin{proof}
	Consider a colour $i\in\lbrace 1,\dots,n\rbrace$ and choose an arbitrary edge $e\in E_i(\Gamma)$. Define $\mathrm{p}(e)=1$ and for each $e_1\in E_i(\Gamma)$ in the same equivalence class define $\mathrm{p}(e_1)=\mathrm{parity}(e,e_1)$.
	Assume that $e_1,e_2\in E_i(\Gamma)$ are two edges in the same equivalence class as $e$. By definition $\mathrm{p}(e_1)\mathrm{p}(e_2)=\mathrm{parity}(e,e_1)\mathrm{parity}(e,e_2)$
	Now, consider the simple cycle from $e_1$ to $e_2$ formed by concatenating the path from $e_1$ to $e$ and then from $e$ to $e_2$. The parity between $e_1$ and $e_2$ is:
	$$\mathrm{parity}(e,e_1)\mathrm{parity}(e,e_2)=\mathrm{parity}(e_1,e_2)$$
	Thus, the desired identity holds, and the function $\mathrm{p}$ is well-defined on each equivalence class.
\end{proof}
\subsection{Proof of Theorem 1.1}
Using Corollary \ref{CorollaryArtinColor} we are ready to prove Theorem \ref{Theorem 1.1}.
\begin{proposition}
	Let $A_\Gamma$ be  an Artin group. Suppose that there exists a prime number $p$ dividing $\frac{l(e)}{2}$ for every $e\in E(\Gamma)$ with even label $l(e)>2$. Then, $A_\Gamma$ satisfies the $\Sigma^1$-conjecture.
\end{proposition}
\begin{proof}
	Assume that we are in the conditions of Corollary \ref{CorollaryArtinColor}. Since $\Lambda$ has no cycles of odd length it is bipartite. Write  $\Lambda=\Lambda_1\sqcup\Lambda_2$. By hypothesis $\Lambda$ admits a  $1$-balanced colouring $\mathcal{C}$ with associated prime number $p$. Define $\varphi:A_\Gamma\to C_p$ as $\varphi(v)=g$ if $v\in V(\Lambda_1)$ and $\varphi(v)=1$ otherwise. Then, $\varphi(uv)=g$ for every $e=\lbrace u,v\rbrace\in E(\Lambda)$ and so, by Corollary \ref{CorollaryArtinColor} $[\chi]\notin\Sigma^1(A_\Gamma)$. Proposition \ref{Proposition 2.9} proves the $\Sigma^1$-conjecture for $A_\Gamma$.
\end{proof}
This result was already proved in \cite[Theorem 1.3]{Escartin} by Escartín-Martínez.
\begin{theorem}\label{Theorem 4.8} The $\Sigma^1$-conjecture is true for balanced Artin groups.
\end{theorem}
\begin{proof}
	Let $A_\Gamma$ be a balanced Artin group and assume that we are in the conditions of Corollary \ref{CorollaryArtinColor}. Since $A_\Gamma$ is balanced $\Lambda$ is balanced, so it admits an $n$-balanced colouring $\mathcal{C}$ with associated primes $p_1,\dots,p_n$ together with a parity map $\mathrm{p}:E(\Lambda)\to\lbrace 1,-1\rbrace$ from Lemma \ref{ParityLemma}. Set $G=C_{p_1}\times\dots\times C_{p_n}$ and  write $g_i$ for the standard generator of $C_{p_i}$. We now define $\varphi:A_\Gamma\to G$. If $v\notin V(\Lambda)$ set $\varphi(v)=1$. To define $\varphi$ on the vertices of $\Lambda$ choose a rooted spanning tree $T\subset\Lambda$ with root $v$ and set $\varphi(v)=1$. Now, for each edge $e=\lbrace u,w\rbrace\in E(\Lambda)$ we declare $\varphi(uw)=g_{\mathcal{C}(e)}^{\mathrm{p}(e)}$. By induction on the depth of $T$ one can easily compute the value of $\varphi(u)$ for every $u\in V(\Lambda)$. We claim that $\varphi(uw)=g_{\mathcal{C}(e)}^{\mathrm{p}(e)}$ for every $e=\lbrace u,w\rbrace\in E(\Lambda\setminus T)$.
	
	Let $e=\lbrace u,w\rbrace$ be an edge of $\Lambda\setminus T$ and form the graph $T_1=T\cup\lbrace e\rbrace$. Since $T$ is a tree, $T_1$ contains exactly one simple cycle $\mathcal{P}$. Label its edges as $e_1,\dots,e_{r}$ in order by walking from $u$ to $w$. It is straightforward to check that:
	$$\varphi(uw)=\prod_{i=1}^{r}g_{\mathcal{C}(e_i)}^{(-1)^{i-1}\mathrm{p}(e_i)}$$
	Because $A_\Gamma$ is balanced each colour appears either $2$ or $0$ times in $\lbrace e,e_1,\dots,e_r\rbrace$. Let $k\in\lbrace 1,\dots,r\rbrace$ be such that $\mathcal{C}(e_k)=\mathcal{C}(e)$. The parity map ensure that pair of edges of the same colour give rise to factors that pairwise cancel in the above formula. Hence, $\varphi(uw)=g_{\mathcal{C}(e_k)}^{(-1)^{k-1}\mathrm{p}(e_k)}$. The number of edges between $e$ and  $e_k$ in $\mathcal{P}$ equals to $k-1$. Therefore:
	$$(-1)^{k-1}\mathrm{p}(e_k)=\mathrm{parity}(e,e_k)\mathrm{p}(e_k)=\mathrm{p}(e)\mathrm{p}(e_k)^2=\mathrm{p}(e)$$
	Thus, $\varphi(uw)=g_{\mathcal{C}(e)}^{\mathrm{p}(e)}$ and we are in the situation of Corollary \ref{CorollaryArtinColor}. Hence, Proposition \ref{Proposition 2.9} proves the $\Sigma^1$-conjecture for $A_\Gamma$.
\end{proof}
As a corollary we obtain a new proof of the $\Sigma^1$-conjectures for the family considered in  \cite{Kochloukova} as well as for the family of coherent Artin groups.
\begin{corollary}
	The $\Sigma^1$-conjecture holds for any Artin group $A_\Gamma$ in either of the following two cases:
	\begin{enumerate}
	\item $\Gamma$ is even and every closed reduced path whose labels exceed $2$ has odd length;
	\item $A_\Gamma$ is coherent.
	\end{enumerate}
\end{corollary}
\begin{proof} In case $1.$ $\Gamma_{>2}^{\mathrm{even}}$ contains no even-length cycles by hypothesis, so is trivially balanced.
	
In case $2.$ any cycle in $\Gamma_{>2}^{\mathrm{even}}$ would lift to a closed reduced path with all labels $\geq 3$ in $\Gamma$. Since $\Gamma$ is chordal there must be three consecutive vertices of the path which form a triangle, but at most one of those edges has a label $> 2$, which is a contradiction. Hence $\Gamma_{>2}^{\mathrm{even}}$ has no closed loops, so is trivially balanced.
\end{proof}
\section{The $\Sigma^2$-conjecture for Artin groups}
In this section, we investigate the converse of Theorem \ref{Theorem 1.2 }. Specifically, we aim to show that under certain conditions, the three conditions stated in that theorem are not only sufficient but also necessary for a character $[\chi]$ to lie in $\Sigma^2(A_\Gamma)$. As a consequence, we establish the $\Sigma^2$-conjecture for the classes of 2-dimensional and coherent Artin groups.
\subsection{Necessity of condition 1}
Fix an Artin group $A_\Gamma$ satisfying the $K(\pi,1)$-conjecture and a character $\chi:A_\Gamma\to\mathbb{Z}$. We begin by proving a technical lemma about cyclotomic polynomials:
\begin{lemma}\label{Lemma 4.5}
	Let $\Phi_n(x)$ denote the $n$-th cyclotomic polynomial. If $p$ is a prime number such that $p\equiv 1\mod n$ then $\Phi_n(x)$ has a root in the finite field $\mathbb{F}_p$. In particular, for every $n$ there exists some prime number $p$ such that $\Phi_n(x)$ has a root in $\mathbb{F}_p$. 
\end{lemma}
\begin{proof}
	Let $w$ be a generator of the multiplicative group $\mathbb{F}_p^\times$, which is cyclic of order $p-1$. Since $p\equiv 1\mod n$ we can set $\xi=w^{\frac{p-1}{n}}$. By construction $\xi$ is a primitive $n$-th root of unity, so $\Phi_n(\xi)=0$. The existence of a prime $p\equiv 1\mod n$ follows from Dirichlet's Theorem on arithmetic progressions.
\end{proof}
\begin{proposition}
	Let $[\chi]\in\Sigma^2(A_\Gamma,\mathbb{Z})$ and let $e=\lbrace v,w\rbrace\in E(\Gamma)$ be an edge such that either $\chi(v)=\chi(w)=0$ or $e$ is dead. Then, there exists a vertex $u\in V(\Gamma)$ such that $A_{v,w,u}$ is spherical. Moreover, if the label $l(e)$ is even, then $u$ can be chosen to be a living vertex.
\end{proposition}
\begin{proof}
	Suppose, by contradiction, that no vertex $u$ makes $A_{v,w,u}$ spherical and assume first that $l(e)$ is odd. Let $p$ be a prime number dividing $l(e)$ and choose a prime number $q$ such that the cyclotomic polynomial $\Phi_{2p}$ has a root in  $\mathbb{F}_q$, which exists due to Lemma \ref{Lemma  4.5}. Define a map $\varphi:A_\Gamma\to \mathbb{F}_q^{\times}$ as follows: set $\varphi(v)$ to be a root of $\Phi_{2p}$, extend $\varphi$ to all vertices connected to $v$ via paths of odd-length by $\varphi(a)=\varphi(v)$ and set $\varphi(b)=1$ elsewhere.

	Now, define the map  $\psi=(\chi,\varphi):A_\Gamma\to \mathbb{Z}\times G$. A direct computation using Formula \ref{Formula3} shows that $\partial(\sigma_e^\psi)=\Phi_{2p}(\varphi(v))[\sigma_w^\psi-\sigma_v^\psi]=0$.
	Moreover, this $2$-cycle cannot lie in  $\mathrm{Im}(\partial_3)$, so it contributes a free $\mathbb{F}[\mathbb{Z}\times G]$-summand to $H_2(\ker(\psi),\mathbb{F})$, which is therefore infinitely generated. Thus $[\chi]\notin\Sigma^2(A_\Gamma,\mathbb{Z})$ by Proposition \ref{Lemma 3.2}.

	Now consider the case where $l(e)$ is even. Suppose by contradiction, that every vertex $u\in V(\Gamma)$ forming a spherical triple $\lbrace u,v,w\rbrace$ is dead. Let $u_1,\dots, u_l\in V(\Gamma)$ be all such vertices (possibly none).	By Formulas \ref{Formula2}--\ref{Formula5}, the 2-cycles $\sigma^\chi_{vw}$ and $\sigma^\chi_{vwu_i}$ lie in $\Ker(\partial)$ over any field if $l(e)=2$ and over fields over characteristic dividing $\frac{l(e)}{2}$ if $l(e)>2$. In either case, these cycles generate and infinitely generated free summand in $H_2(\ker(\chi),\mathbb{F})$, again implying that $[\chi]\notin\Sigma^2(A_\Gamma)$ by Proposition \ref{Lemma 3.2}.
\end{proof}
\begin{corollary}\label{Corollary 5.6}
	Condition 1. is necessary for $[\chi]$ to lie in $\Sigma^2(A_\Gamma)$ in the case where $A_\Gamma$ satisfies the $K(\pi,1)$-conjecture and $\Gamma$ has no triangles of the form $\lbrace 2,3,3\rbrace,\lbrace 2,3,4\rbrace,\lbrace 2,3,5\rbrace$ or $\lbrace 2,2,k\rbrace$ with $k$ odd.
\end{corollary}
Even Artin groups satisfying the $K(\pi,1)$-conjecture and $2$-dimensional Artin groups, which satisfy the $K(\pi,1)$-conjecture (c.f. \cite{Charney-Davis}), are two relevant families satisfying the Corollary. 
\subsection{Necessity of condition 2}
Let $A_\Gamma$ be an arbitrary Artin group satisfying the $K(\pi,1)$-conjecture and let $\chi:A_\Gamma\to\mathbb{Z}$ be a character. The non-empty part follows trivially:
\begin{lemma}\label{Lemma  4.2}
	If $[\chi]\in\Sigma^2(A_\Gamma,\mathbb{Z})$ (or $[\chi]\in\Sigma^2(A_\Gamma)$) then for any dead vertex $v\in V(\Gamma)$ the spherical link $\mathrm{slk}_{\mathrm{Liv}^\chi,v}^\chi$ is non-empty.
\end{lemma}
\begin{proof}
	 Since $[\chi]\in\left(\Sigma^2(A_\Gamma,\mathbb{Z})\cup\Sigma^2(A_\Gamma)\right)\subset\Sigma^1(A_\Gamma)$ Proposition \ref{Proposition 1.5}.2 implies that $\text{Liv}^\chi$ is dominant in $\Gamma$. In particular, every dead vertex is adjacent to at least one living vertex, hence $\mathrm{slk}_{\mathrm{Liv}^\chi,v}^\chi\neq\emptyset$.
\end{proof}
While the connectedness of spherical links cannot be established in general, we can prove it under an additional condition on the edge labels:
\begin{proposition} Let $[\chi]\in\Sigma^2(A_\Gamma,\mathbb{Z})$. Assume that $v\in  V(\Gamma)$ is a dead vertex such that there exists a prime number $p$ dividing $\frac{l(e)}{2}$ for every dead edge $e\in E(\mathrm{slk}_{\mathrm{Liv}^\chi,v}^\chi)$. Then, the spherical link $\mathrm{slk}_{\mathrm{Liv}^\chi,v}^\chi$ is connected.
\end{proposition}
\begin{proof}
	Let $\mathbb{F}$ be a field of characteristic $p$ and assume by contradiction that $\mathrm{slk}_{\mathrm{Liv}^\chi,v}^\chi$ is not connected. Let $\Delta_1$ and $\Delta_2$ be two connected components of $\mathrm{slk}_{\mathrm{Liv}^\chi,v}^\chi$ and fix $u\in V(\Delta_1)$ and $w\in  V(\Delta_2)$. From Formulas \ref{Formula2} and \ref{Formula3}, for any $u,w\in V(\mathrm{slk}_{\mathrm{Liv}^\chi,v}^\chi)$ there exist non-zero Laurent polynomials $p_u(t),p_w(t)\in\mathbb{F}[t^{\pm1}]$ such that $p_u(t)\sigma_{uv}^\chi-p_w(t)\sigma_{wv}^\chi\in\ker(\partial_2)$.
	 
	However, since $\Delta_1$ and $\Delta_2$ are disconnected, there is no triple  $\lbrace u',w',v\rbrace$ with $u'\in V(\Delta_1),v'\in V(\Delta_2)$ for which $A_{u',w',v}$ is spherical. Thus, the difference $\sigma_{uv}^\chi-\sigma_{wv}^\chi$ cannot appear in the image of $\partial_3$. As a result, the cycle $p_u(t)\sigma_{uv}^\chi-p_w(t)\sigma_{wv}^\chi\in\ker(\partial_2)$ survives in the homology and generates an infinite-dimensional summand in $H_2(A_\Gamma,\mathbb{F})$, contradicting the assumption that $[\chi]\in\Sigma^2(A_\Gamma,\mathbb{Z})$ by Proposition  \ref{Proposition 2.4}.
\end{proof}
\begin{corollary}\label{2-dimensional.2}
	Condition $2.$ is a necessary condition for the $\Sigma^2$-conjecture in:
	\begin{enumerate}
		\item Artin groups satisfying the $K(\pi,1)$-conjecture for which there exists a prime number $p$ dividing $\frac{l(e)}{2}$ for every edge $e\in E(\Gamma)$ with $l(e)>2$ even.
		\item $2$-dimensional Artin groups.
	\end{enumerate} 
\end{corollary}
The first part of the Corollary recovers a known result, since the $\Sigma^2$-conjecture is known for this family (see Escartín–Martínez \cite{Escartin}, Theorem 7.4).
\subsection{The $2$-dimensional case}
Since $2$-dimensional Artin groups have cohomological dimensional $2$, the BNSR invariants stabilize:
 $$\Sigma^n(A_\Gamma)=\Sigma^m(A_\Gamma)~\forall~n,m\geq 2$$ 
 
Therefore, to prove Theorem \ref{Theorem 1.4} it suffices to compute $\Sigma^2$.
\begin{proof}(Theorem \ref{Theorem 1.4}) 
	By Corollaries \ref{Corollary 5.6} and \ref{2-dimensional.2} Conditions 1. and 2. are necessary. It remains to prove condition $3$. Since $A_\Gamma$ is $2$-dimensional, $\partial_n=0$ for all $n\geq 3$ and thus $H_2(A_\Gamma,\mathbb{F})=\ker(\partial_2)$ for any field $\mathbb{F}$. In particular, $\dim_\mathbb{F}H_2(A_\Gamma,\mathbb{F})=\infty$ if and only if $\ker(\partial_2)\neq0$.
	
Suppose, by way of contradiction, that $\mathrm{Liv}^\chi$ is not a tree. Then, it contains a non-trivial reduced cycle. Let  $\mathcal{P}$ be a closed reduced path in $\Gamma$, consisting of edges $e_1,\dots,e_n$ and vertices $v_1,\dots,v_n$, where $e_i=\{v_i,v_{i+1}\}$ and $v_{n+1}=v_1$. By Formulas \ref{Formula2} and \ref{Formula3}, for each $i$ there exists $p_i,q_i\in\mathbb{F}[t^{\pm1}]\setminus\lbrace0\rbrace$ such that $\partial(\sigma_{v_i,v_{i+1}}^\chi)=p_{i+1}(t)\sigma_{v_{i+1}}^\chi-q_i(t)\sigma_{v_i}^\chi$. We aim to show that there exist coefficients $c_1,\dots,c_n\in\mathbb{F}[t^{\pm1}]\setminus\lbrace0\rbrace$ such that $\sum\limits_{i=1}^nc_i\sigma_{v_iv_{i+1}}\in\ker(\partial_2)$.
\begin{claim*}
Such coefficients $c_i$ exist if and only if  $\prod\limits_{i=1}^np_i(t)=\prod\limits_{j=1}^nq_j(t)$.
\end{claim*}
\begin{proof}
	Observe that $\partial_2\left(\sum\limits_{i=1}^nc_i\sigma_{v_iv_{i+1}}\right)=(c_np_1(t)-c_1q_1(t))\sigma_{v_1}^\chi+\sum\limits_{i=2}^n(c_{i-1}p_i(t)-c_iq_i(t))\sigma_{v_i}=0$. This vanishes if and only if the following system of equations hold:
	$$c_np_1(t)-c_1q_1(t)=0\text{ and }c_{i-1}p_i(t)-c_iq_i(t)=0~\forall~i=2,\dots,n$$
	
	This system of equation has a solution if and only if  $\prod\limits_{i=1}^np_i(t)=\prod\limits_{j=1}^nq_j(t)$. Hence, the claim follows.
	\phantom{\qedhere}
\end{proof}	
	It remains to verify this identity. From Formulas \ref{Formula2} and \ref{Formula3}, we know that if $l(e_i)$ is odd then $p_{i+1}(t)=q_i(t)$ and if $l(e_i)$ is even there exists $r_i\in\mathbb{F}[t^{\pm1}]$ such that $p_{i+1}(t)=(1-t^{\chi(v_i)})r_i(t),q_{i}(t)=(1-t^{\chi(v_{i+1})})r_i(t)$. Hence:
	
	$$\prod\limits_{i=1}^np_i(t)=\prod\limits_{l(e_i)~\mathrm{ odd}}p_{i+1}(t)\prod\limits_{l(e_i)~\mathrm{ even}}(1-t^{\chi(v_i)})r_i(t)$$
	$$\prod\limits_{i=1}^nq_i(t)=\prod\limits_{l(e_i)~\mathrm{ odd}}q_{i}(t)\prod\limits_{l(e_i)~\mathrm{ even}}(1-t^{\chi(v_{i+1})})r_i(t)$$
	Now, since the path $\mathcal{P}$ is a closed loop and the character $\chi$ is constant along edges with odd labels, we have
	$$\prod\limits_{l(e_i)~\mathrm{ even}}(1-t^{\chi(v_{i+1})})=\prod\limits_{l(e_i)~\mathrm{ even}}(1-t^{\chi(v_{i})})$$
	so the products match, and hence we are done.
\end{proof}
As a consequence, we can compute the finiteness properties of the derived subgroups of 2-dimensional Artin groups.
\begin{proof}(Corollary \ref{Corollary 1.5})
	If $\Gamma$ is an odd tree it follows by Theorem \ref{Theorem 1.4} that $\Sigma^2(A_\Gamma)=S(A_\Gamma)$. Assume now that $\Gamma$ is not an odd tree. There are two possibilities:
	\begin{itemize}
		\item If $A_\Gamma/A_\Gamma'\cong\mathbb{Z}$: $S(A_\Gamma)=\lbrace[\chi],[-\chi]\rbrace$, where $\chi:A_\Gamma\to\mathbb{Z}$ is the character sending all standard generators to $1$. 	For both characters we have $\textrm{Liv}^\chi=\Gamma$.
		But then $\Gamma$ is not a tree, because if it is a tree and the abelianization has rank $1$ then all the labels must be odd. Hence, $\Sigma^2(A_\Gamma)=\emptyset$.
		\item If $A_\Gamma/A_\Gamma'\ncong\mathbb{Z}$: in this case there must exists an edge $e=\lbrace u,v\rbrace\in E(\Gamma)$ such that the images of $u$ and $v$ in the abelianization are not equal. This implies that $l(e)$ must be even. Now, for any character $\chi:A_\Gamma\to\mathbb{Z}$ such that $\chi(u)=-\chi(v)$ the edge $e$ is $2$-dead  and so $[\chi]\notin\Sigma^2(A_\Gamma)$, by Theorem \ref{Theorem 1.4}.\qedhere
	\end{itemize}
\end{proof}
\subsection{Coherent Artin groups}
If $A_\Gamma$ is a coherent  Artin group then $\Sigma^1(A_\Gamma)=\Sigma^2(A_\Gamma,\mathbb{Z})$. Thus, since coherent Artin groups satisfy the $\Sigma^1$-conjecture, to prove the $\Sigma^2$-conjecture it suffices to show that $\chi:A_\Gamma\to\mathbb{Z}$ satisfies the conditions of Theorem \ref{Theorem 1.2 } if and only if $\mathrm{Liv}^\chi$ is connected and dominant.

One direction of this equivalence holds for all Artin groups.
\begin{lemma}
	Let $A_\Gamma$ be an Artin group and $\chi:A_\Gamma\to\mathbb{Z}$ a character. If $\chi$ satisfies the conditions of Theorem \ref{Theorem 1.2 }, then $\mathrm{Liv}^\chi$ is connected and dominant.
\end{lemma}
\begin{proof}
	If $\mathrm{Liv}^\chi$ is connected, then condition $3.$ of Theorem \ref{Theorem 1.2 } does not hold, as the associated simplicial complex would not be connected. Moreover, since $\mathrm{slk}_{\mathrm{Liv}^\chi,v}^\chi\neq\emptyset$ for every dead vertex $v\in V(\Gamma)$, it follows that $\mathrm{Liv}^\chi$ is dominant.
\end{proof}
To prove the converse, we need the coherence assumption. We start with a technical Lemma.
\begin{lemma}\label{Lemma 5.8} Let $A_\Gamma$ be a coherent Artin group and let $\chi:A_\Gamma\to\mathbb{Z}$ be a character such that $\mathrm{Liv}^\chi$ is connected. If $v\in V(\Gamma)$ is a dead vertex with $\mathrm{slk}_{\mathrm{Liv}^\chi,v}^\chi\neq\emptyset$ and if $u,w\in V(\mathrm{slk}_{\mathrm{Liv}^\chi,v}^\chi)$ lie in different connected components, then $\lbrace u,w\rbrace\notin E(\Gamma)$.
\end{lemma}
\begin{proof}
	Suppose by contradiction that $\lbrace u,w\rbrace\in E(\Gamma)$. Then, the subgraph induced by $\lbrace u,v,w\rbrace$ is equal to:
		$$
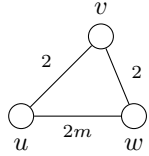
\begin{tikzpicture}[main/.style = {draw, circle},node distance={15mm}] 
		\node[label=below:{$u$}][main] (1)  {}; 
		\node[label={$v$}][main] (6) [above right of=1] {};
		\node[label=below:{$w$}][main] (2) [right of=1] {}; 
		
		\draw[-] (1) -- node[below] {$\scriptstyle{2m}$}  (2);
		\draw[-] (1) -- node[above left]{$\scriptstyle{2}$}  (6);
		\draw[-] (6) -- node[pos=0.7,above right] {$\scriptstyle{2}$}   (2);
	\end{tikzpicture} $$
	for some $m\geq 2$. Since $\mathrm{Liv}^\chi$ is connected there exists a path $\mathcal{P}$ in $\mathrm{Liv}^\chi$ from $u$ to $w$. Let $\Gamma_1$ be the subgraph induced by $\mathcal{P}\cup\lbrace u,w\rbrace$. As this subgraph contains a cycle and $\Gamma$ is chordal, there must exists a vertex $a\in V(\mathcal{P})$ such that $\lbrace u,w,a\rbrace$ forms a triangle in $\Gamma$. In this situation $a$ and $v$ is connected, since otherwise $\lbrace u,v,w,a\rbrace$ would induce a subgraph which is forbidden by definition of coherent Artin group. Hence, $a$ is connected to both $u$ and $w$ in $\mathrm{slk}_{\mathrm{Liv}^\chi,v}^\chi$ which is a contradiction.
\end{proof}

\begin{proposition}
	Let $A_\Gamma$ be a coherent Artin group and $\chi:A_\Gamma\to\mathbb{Z}$ a character. Then, $\mathrm{Liv}^\chi$ is connected and dominant if and only if $\chi$ satisfies the conditions of Theorem \ref{Theorem 1.2 }
\end{proposition}
\begin{proof}
	Assume that $\mathrm{Liv}^\chi$ is connected and dominant. We have to verify the three conditions in Theorem \ref{Theorem 1.2 }. For the first condition we have two cases:
		\begin{itemize}
			\item Let $e=\lbrace v,w\rbrace$ be an edge with $\chi(v)=\chi(w)=0$. Since $\mathrm{Liv}^\chi$ is dominant there exists $v',w'\in V(\Gamma)$ connected to $v$ and $w$ respectively with $\chi(v'),\chi(w')\neq0$. If $v'=w'$ we are done. Otherwise there exists a path $\mathcal{P}$ in $\mathrm{Liv}^\chi$ from $v'$ to $w'$. Consider $\Gamma_1$ to be the subgraph induced by $\mathcal{P}\cup\lbrace v,v',w,w'\rbrace$. Since $\Gamma_1$ contains a closed path containing $v$ and $w$ and $\Gamma$ is chordal there must exists a vertex $u\in V(\Gamma_1)$ such that $\lbrace u,v,w\rbrace$ forms a triangle in $\Gamma$. 
			\item Let $e=\lbrace v,w\rbrace$ be a dead edge. Since $\mathrm{Liv}^\chi$ is connected there is a path $\mathcal{P}$ from $v$ to $w$ in $\mathrm{Liv}^\chi$. Consider $\Gamma_1$ to be the subgraph induced by $\mathcal{P}\cup\lbrace v,w\rbrace$. Since $\Gamma_1$ contains a closed path containing $v$ and $w$ and $\Gamma$ is chordal there must exists a vertex $u\in V(\Gamma_1)$ such that $\lbrace u,v,w\rbrace$ forms a triangle in $\Gamma$ as we wanted to show. 
		\end{itemize}
	For the second condition let $v\in V(\Gamma)$ be a dead edge, then $\mathrm{slk}_{\mathrm{Liv}^\chi,v}^\chi\neq\emptyset$ because $\mathrm{Liv}^\chi$ is dominant. Suppose that $\mathrm{slk}_{\mathrm{Liv}^\chi,v}^\chi$ is disconnected and let $u,w$ be two vertices in different components. By Lemma \ref{Lemma 5.8} $\lbrace u,w\rbrace\notin E(\Gamma)$. Let $\mathcal{P}$ be a path from $u$ to $w$ in $\mathrm{Liv}^\chi$. Consider $\Gamma_1$ to be the subgraph induced by $\mathcal{P}\cup\lbrace u,v,w\rbrace$. Since $\Gamma_1$ contains a closed path and $\Gamma$ is chordal there must be some $a,b\in V(\mathcal{P})$ such that the subgraph induced by $\lbrace u,v,w,a,b\rbrace$ has the following form:
			$$\begin{tikzpicture}[main/.style = {draw, circle},node distance={15mm}] 
				\node[label=below:{$u$}][main] (1)  {}; 
				\node[label={$a$}][main] (6) [above right of=1] {};
				\node[label=below:{$v$}][main] (2) [right of=1] {}; 
				\node[label={below:$w$}][main] (3) [right of=2] {};
				\node[label={$b$}][main] (4) [above left of =3] {};
				
				\draw[-] (1) -- (2);
				\draw[-] (1) -- (6);
				\draw[-] (6) --  (2);
				\draw[-] (2) --   (3);
				\draw[-] (2) --   (4);
				\draw[-] (4) --   (3);
			\end{tikzpicture} $$
			
			There are three possibilities:
			\begin{itemize}
				\item \textbf{Case 1: $a = b$.} Then, both $\{u, a\}$ and $\{w, a\}$ are edges of $\Gamma$. By Lemma~\ref{Lemma 5.8}, it follows that $u$, $a$, and $w$ lie in the same connected component of $\mathrm{slk}_{\mathrm{Liv}^\chi,v}^\chi$, contradicting our assumption.
				
				\item \textbf{Case 2: $\{a, b\} \in E(\Gamma)$.} Since $\{u, a\}$, $\{a, b\}$, and $\{w, b\}$ are all edges in $\Gamma$, Lemma~\ref{Lemma 5.8} implies that $u$, $a$, $b$, and $w$ all lie in the same connected component of $\mathrm{slk}_{\mathrm{Liv}^\chi,v}^\chi$, again a contradiction.
				
				\item \textbf{Case 3: $\{a, b\} \notin E(\Gamma)$.} In this case, we repeat the construction: since $\{u,a,v\}$ and $\{w,b,v\}$ are triangles in $\Gamma$, we may find new vertices $c,d \in V(\mathrm{Liv}^\chi)$ such that $\{a, c, v\}$ and $\{b, d, v\}$ are also triangles. Again, if $c = d$ or $\{c,d\} \in E(\Gamma)$, we reduce to one of the previous two cases and obtain a contradiction. Otherwise, we iterate the process by finding vertices $e,f \in V(\mathrm{Liv}^\chi)$ with $\{c, e, v\}$ and $\{d, f, v\}$ forming triangles in $\Gamma$. Since $\Gamma$ is finite, this process must terminate, and eventually one of the previous two cases will occur, yielding a contradiction.
			\end{itemize}
	The last condition follows by [\cite{Goyal}{ Theorem 3.4}].\qedhere
\end{proof}
Applying [\cite{Lopez} Theorem 3.3] we immediately get that all $\Sigma$-invariants of coherent Artin groups coincide:
\begin{theorem}\label{Theorem 1.6}
	If $A_\Gamma$ is a coherent Artin group, then:
	$$\Sigma^n(A_\Gamma)=\Sigma^n(A_\Gamma,\mathbb{Z})=\lbrace[\chi]\in S(A_\Gamma)\mid\mathrm{Liv}^\chi\text{ is connected and dominant}\rbrace~\forall~n\geq 1$$
\end{theorem}
\section{Algebraic fibring of Artin kernels}
Let $\mathcal{P}$ be a finiteness property such as being $F_n$ or $FP_n$. A group $G$ is said to be \textbf{$\mathcal{P}$-algebraically fibred} if it admits a non-trivial homomorphism $\varphi:G\to\mathbb{Z}$ with $\ker(\varphi)$  of type $\mathcal{P}$.

Using the results about the $\Sigma^1$ and $\Sigma^2$-conjecture along with Proposition \ref{Proposition 2.8} we can construct examples of Artin kernels that  are either finitely generated but not finitely presented, or not even finitely generated.
\begin{example} Let $\Gamma$ be the following graph:
	$$\begin{tikzpicture}[main/.style = {draw, circle},node distance={15mm}] 
		\node[label=below:{$a_2$}][main] (1)  {}; 
		\node[label={$a_1$}][main] (6) [above right of=1] {};
		\node[label=below:{$a_3$}][main] (2) [right of=1] {}; 
		\node[label={below:$b_1$}][main] (3) [right of=2] {};
		\node[label=below:{$b_2$}][main] (4) [right of=3] {}; 
		\node[label=below:{$c_1$}][main] (5) [right of=4] {}; 
		
		\draw[-] (1) -- node[below] {$\scriptstyle{2}$}   (2);
		\draw[-] (1) -- node[pos=0.4,above] {$\scriptstyle{3}$}   (6);
		\draw[-] (6) -- node[pos=0.7,above right] {$\scriptstyle{3}$}   (2);
		
		\draw[-] (2) -- node[below] {$\scriptstyle{4}$}   (3);
		\draw[-] (3) -- node[below] {$\scriptstyle{5}$}   (4);
		\draw[-] (4) -- node[below] {$\scriptstyle{4}$}   (5);
	\end{tikzpicture} $$
	Let $\chi:A_\Gamma\to\mathbb{Z}$ be a character and set $a=\chi(a_1)=\chi(a_2)=\chi(a_3)$, $b=\chi(b_1)=\chi(b_2)$ and $c=\chi(c_1)$. Assume further that $a,b,c+b,a+b\neq 0$. Then, $[\chi]\in\Sigma^2(A_\Gamma)$ and if $\chi$ is discrete, $\Ker(\chi)$ is finitely presented.
\end{example}
\begin{example}
	Let $\Gamma$ be the following graph:
	$$
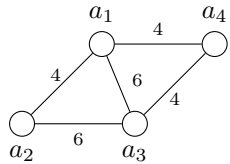
\begin{tikzpicture}[main/.style = {draw, circle},node distance={15mm}] 
		\node[label=below:{$a_2$}][main] (1)  {}; 
		\node[label={$a_1$}][main] (6) [above right of=1] {};
		\node[label=below:{$a_3$}][main] (2) [right of=1] {}; 
		\node[label={$a_4$}][main] (3) [above right of=2] {};
		\draw[-] (1) -- node[below] {$\scriptstyle{6}$}   (2);
		\draw[-] (1) -- node[pos=0.4,above] {$\scriptstyle{4}$}   (6);
		\draw[-] (6) -- node[pos=0.7,above right] {$\scriptstyle{6}$}   (2);
		\draw[-] (6) -- node[above] {$\scriptstyle{4}$}   (3);	
		\draw[-] (2) -- node[below] {$\scriptstyle{4}$}   (3);
	\end{tikzpicture} $$
The graph $\Gamma$ is connected and $\pi_1(\Gamma)$ has rank $2$, so $A_\Gamma$ satisfies the $\Sigma^1$-conjecture. Since $A_\Gamma$ is also $2$-dimensional, it satisfies the $\Sigma^2$-conjecture. Consider two characters:
\begin{itemize}
	\item Define $\chi_1:A_\Gamma\to\mathbb{Z}$ by $\chi_1(a_1)=\chi_1(a_2)\neq 0$ and $\chi_1(a_3)=\chi_1(a_4)=0$. Then, $[\chi]\in\Sigma^1(A_\Gamma)\setminus\Sigma^2(A_\Gamma)$, so $\Ker(\chi)$ is finitely generated but not finitely presented.
	\item Define $\chi_2:A_\Gamma\to\mathbb{Z}$ by $\chi_2(a_1)=\chi_2(a_2)=0$ and $\chi_2(a_3)=\chi_2(a_4)\neq0$. In this case  $[\chi]\notin \Sigma^1(A_\Gamma)$, so $\Ker(\chi)$ is not finitely generated.
\end{itemize} 

\end{example}

\end{document}